\documentclass[11pt,reqno]{amsart}

\usepackage[margin=1.1in]{geometry}
\usepackage{amsmath,amssymb,mathtools,amsthm,stmaryrd,mathrsfs}
\usepackage{booktabs,array,microtype,aliascnt}

\usepackage{hyperref}
\usepackage{xcolor}
\hypersetup{
	colorlinks,
	linkcolor={red!50!black},
	citecolor={blue!50!black},
	urlcolor={blue!80!black}
}
\usepackage[capitalize,noabbrev]{cleveref}

\numberwithin{equation}{section}
\allowdisplaybreaks

\newtheorem{theorem}{Theorem}[section]
\newaliascnt{proposition}{theorem}
\newtheorem{proposition}[proposition]{Proposition}
\aliascntresetthe{proposition}
\newaliascnt{lemma}{theorem}
\newtheorem{lemma}[lemma]{Lemma}
\aliascntresetthe{lemma}
\newaliascnt{corollary}{theorem}
\newtheorem{corollary}[corollary]{Corollary}
\aliascntresetthe{corollary}
\newaliascnt{definition}{theorem}

\aliascntresetthe{definition}
\theoremstyle{remark}
\newaliascnt{remark}{theorem}
\newtheorem{remark}[remark]{Remark}
\aliascntresetthe{remark}
\newtheorem*{maintheorem}{Main Theorem}

\crefname{theorem}{Theorem}{Theorems}
\Crefname{theorem}{Theorem}{Theorems}
\crefname{proposition}{Proposition}{Propositions}
\Crefname{proposition}{Proposition}{Propositions}
\crefname{lemma}{Lemma}{Lemmas}
\Crefname{lemma}{Lemma}{Lemmas}
\crefname{corollary}{Corollary}{Corollaries}
\Crefname{corollary}{Corollary}{Corollaries}
\crefname{definition}{Definition}{Definitions}
\Crefname{definition}{Definition}{Definitions}
\crefname{remark}{Remark}{Remarks}
\Crefname{remark}{Remark}{Remarks}

\newcommand{\kk}{\Bbbk}

\newcommand{\ZZ}{\mathbb{Z}}
\newcommand{\RR}{\mathbb{R}}

\newcommand{\cA}{\mathcal{A}}

\newcommand{\cF}{\mathcal{F}}
\newcommand{\cL}{\mathcal{L}}

\newcommand{\cU}{\mathcal{U}}

\newcommand{\SR}{\mathscr S}
\newcommand{\UR}{\mathscr U}

\newcommand{\Asup}{\mathcal A_s^{\mathrm{up}}}
\newcommand{\Uneg}{\mathcal U_u^{-}}
\newcommand{\Mu}{\mathcal M_u}
\newcommand{\Muwt}[1]{\mathcal M_{u,#1}}
\newcommand{\Ugen}{\mathscr U_{\mathrm{gen}}}
\newcommand{\udisc}{u_{\Delta}}

\newcommand{\wt}{\operatorname{wt}}
\newcommand{\Frac}{\operatorname{Frac}}
\newcommand{\Spec}{\operatorname{Spec}}
\newcommand{\Sym}{\operatorname{Sym}}
\newcommand{\GL}{\operatorname{GL}}
\newcommand{\SL}{\operatorname{SL}}
\newcommand{\Rep}{\operatorname{Rep}}
\newcommand{\SI}{\operatorname{SI}}
\newcommand{\ord}{\operatorname{ord}}
\newcommand{\initer}{\operatorname{in}}
\newcommand{\dlog}{\,d\!\log}
\newcommand{\pos}[1]{\left[#1\right]_{+}}

\newcommand{\Schur}{\mathbf{S}}
\newcommand{\Specht}{\mathsf{S}}

\title{A Polyhedral Formula for $n\times2\times2$ Kronecker Coefficients via Cluster Algebras}
\author{Jiarui Fei}
\address{School of Mathematical Sciences, Shanghai Jiao Tong University, Shanghai, China}
\email{jiarui@sjtu.edu.cn}
\author{Chenxin Xue}
\address{School of Mathematical Sciences, Shanghai Jiao Tong University, Shanghai, China}
\email{xcx0711@sjtu.edu.cn}
\thanks{The authors were supported in part by the National Natural Science Foundation of China (Nos.~12131015 and 12571038).}
\date{}

\begin{document}

\begin{abstract}
Let
\[ \mathscr U=\Bbbk[\Bbbk^3\otimes\Bbbk^2\otimes\Bbbk^2]^{U_3\times U_2\times U_2}. \]
We construct an ordinary cluster family with Markov principal part.  At
$\zeta=-1$, the
intersection of its initial Laurent ring with the three adjacent Laurent
rings equals $\mathscr U$, and
\[ \mathscr U=\mathcal M_u[u_\Delta], \]
where $\mathcal M_u$ is the middle algebra.  Its theta functions are indexed
by a cone with a sixteen-element Hilbert basis.  Multiplication by
$u_\Delta$ pairs the Hilbert generators of mutable degrees $1$ and $-1$ and
reduces each triple-weight space to the slice $\ell=\ell_0$.  Counting the
lattice points in this slice gives a finite sum with nonnegative summands.
Determinant reduction extends the formula to all $n\times2\times2$
Kronecker coefficients.
\end{abstract}

\maketitle

\section*{Introduction}

For partitions $\lambda,\mu,\nu$ of the same integer, the Kronecker
coefficient $g_{\lambda,\mu,\nu}$ is the multiplicity of
$\Specht^\lambda$ in $\Specht^\mu\otimes\Specht^\nu$.  No positive
combinatorial rule is known in general.  When $\mu$ and $\nu$ have at most
two parts, however, several formulas and structural descriptions are
available: the formulas of Remmel--Whitehead and Rosas \cite{RW,Rosas}, the
geometric construction of Adsul--Subrahmanyam \cite{AS}, the crystal basis
of Blasiak--Mulmuley--Sohoni \cite{BMS}, the chamber quasipolynomials of
Briand--Orellana--Rosas \cite{BOR09}, and the vector-partition formulation
of Mishna--Rosas--Sundaram \cite{MRS}.  Explicit computation in this range
is therefore already available.  Here a cluster theta basis gives a
different answer: one canonical slice of its tropical cone, written as a
single sum of integer-interval counts.

Call a triple $\chi=(\lambda;\mu;\nu)$ \emph{admissible}
if the three partitions have the same size and
$\ell(\lambda)\le3$, $\ell(\mu),\ell(\nu)\le2$.  Write
\[ \chi=(\lambda_1,\lambda_2,\lambda_3;\mu_1,\mu_2;\nu_1,\nu_2) \]
and set
\[ A=\lambda_1+2\lambda_2-\mu_1-\nu_1,
   \qquad
   \ell_0=\max\left\{0,\left\lfloor\frac{A+1}{2}\right\rfloor\right\}. \]
Let $\chi_\Delta=(2,2,0;2,2;2,2)$ and write
\[ \chi-\ell_0\chi_\Delta
   =(L_1,L_2,L_3;M_1,M_2;N_1,N_2). \]
If this shifted weight is not dominant, then $g_{\lambda,\mu,\nu}=0$.
Otherwise interchange $\mu$ and $\nu$, if necessary, so that $M_2\ge N_2$,
and put
\[
 \begin{aligned}
 a&=L_1-L_2,& b&=L_2-L_3,& c&=L_3,\\
 \rho&=M_1-M_2,& h&=M_2-N_2,&
 d&=M_1+N_1-L_1-2L_2.
 \end{aligned}
\]
For $x\in\ZZ$, define
\[
 \begin{aligned}
 m_x&=b-h-2x,\\
 \mathsf L(x)&=\max\left\{
 0,\left\lceil\frac{a-d-x}{2}\right\rceil,N_2-2c-x
 \right\},\\
 \mathsf U(x)&=\min\{0,m_x\}
 +\min\left\{
 a,\left\lfloor\frac{a+h+x}{2}\right\rfloor,M_2-L_2+x
 \right\}.
 \end{aligned}
\]
Then
\[ g_{\lambda,\mu,\nu}
   =\sum_{x=0}^{\min\{b,\rho\}}
     \pos{\mathsf U(x)-\mathsf L(x)+1}. \]
Here $[r]_+=\max\{r,0\}$.  For fixed $x$, the summand is the number of
integers in $[\mathsf L(x),\mathsf U(x)]$.  The formula has no cancellation
and requires no chamber decomposition in its statement.  Only division by
two occurs in the endpoints, so the resulting chamber quasipolynomials have
quasiperiod dividing two; period two occurs.  Rectangular translation and
stabilization extend the same expression to every $n\times2\times2$
Kronecker coefficient; see \cref{cor:closed-formula-22n}.

The invariant-theoretic source of the formula is
\[ \UR=\kk[\kk^3\otimes\kk^2\otimes\kk^2]^{U_3\times U_2\times U_2}. \]
Its component of weight $(\lambda;\mu;\nu)$ has dimension
$g_{\lambda,\mu,\nu}$.  We construct an ordinary cluster family with Markov
principal part and coefficient $\zeta$.  Let $\Uneg$ be the intersection,
after specialization at $\zeta=-1$, of the initial Laurent ring and its
three adjacent Laurent rings.  Let $\Mu$ be the algebra spanned by the
specialized theta functions regular along the four frozen boundary
divisors, let $\Xi\subset\RR^7$ be their tropical cone of regularity, and let
$W$ be the weight matrix of the seed.

\begin{maintheorem}
The following statements hold.
\begin{enumerate}
\item The thirteen invariants in \eqref{eq:thirteen-invariant-weights}
generate $\UR$, and
\[ \UR=\Uneg=\Mu[\udisc],
   \qquad \udisc=u_{220,22}^{22}. \]
\item The algebra $\Mu$ has theta basis
\[ \{\theta_v\mid v\in\Xi\cap\ZZ^7\}. \]
The sixteen labels in \eqref{eq:Xi-Hilbert-basis} form the Hilbert basis of
$\Xi\cap\ZZ^7$, and the corresponding theta functions form a Khovanskii
basis of $\Mu$.
\item For every admissible weight $\chi$,
\[ \UR_\chi
   =\udisc^{\ell_0}\Muwt{\chi-\ell_0\chi_\Delta}, \]
and hence
\[ g_{\lambda,\mu,\nu}
   =\#\{v\in\Xi\cap\ZZ^7\mid
          vW=\chi-\ell_0\chi_\Delta\}. \]
Integral coordinates on this two-dimensional slice give the finite sum
above.
\end{enumerate}
\end{maintheorem}

The cluster structure supplied by invariant theory first appears on the
double-$U$-invariant algebra
\[ \SR=\kk[\kk^3\otimes\kk^2\otimes\kk^2]^{U_3\times U_2}. \]
The last root subgroup acts on $\SR$ through a locally nilpotent derivation
$\partial$, and $\UR=\ker\partial$.  One quotient exchange has a negative
sign.  Introducing the coefficient $\zeta$ places the three quotient
exchanges in an ordinary cluster family over $\kk[\zeta^{\pm1}]$; the
signed relations are recovered at $\zeta=-1$.

A local slice for $\partial$ gives rational coordinates on
$Y_{322}=\Spec\UR$.  Contracting the logarithmic top form on the upstairs
seed gives
\[ \omega_{322}=-\frac{dq_1\wedge\cdots\wedge dq_7}{q_1q_2^3}. \]
Together with two quotient identities, this form detects adjacent
logarithmic charts with frozen divisor $u_2u_3u_4u_{13}=0$ and exchange
\[ u_7u_8=u_4u_5^2+u_1^2u_{13}. \]
The same data have a $D_5$-parabolic interpretation, recorded in
Appendix~\ref{app:D5-parabolic}.

The proof that the thirteen invariants generate $\UR$ uses two local
slices.  They identify $\UR$ with the proposed subalgebra after localization
at $u_1$ and at $u_2$; one regular-sequence calculation removes the
localizations.  Factorization in the UFD $\UR$ then gives
$\Uneg\subseteq\UR$.

The mutable quiver is the Markov quiver of the once-punctured torus.
Zhou's theorem for the principal-coefficient Markov seed \cite{Zhou}
gives finite theta functions, and the specialization results of
\cite{CMMM} preserve their tropicalizations and valuative independence at
$\zeta=-1$.  Li's middle--upper theorem for surfaces with at least two
punctures \cite{Li} gives a related result, but does not apply to the present
once-punctured fiber with noninvertible frozen variables.

Regularity along the four boundary divisors cuts out $\Xi$.  Its Hilbert
basis contains four generators of mutable degree $1$, four of degree $-1$,
and eight of degree $0$.  The negative-degree generators are scalar
multiples of $\udisc$ times their positive-degree partners.  This pairing
moves every weight space to the single level $\ell_0$.  The resulting slice
has dimension two, and projection to one coordinate gives the interval sum
above.

Appendix~\ref{app:partial-ring} recalls the cluster algebra on $\SR$.
Appendix~\ref{app:regular-sequence-certificate} gives the regular-sequence
calculation.  Appendix~\ref{app:folded-DT} identifies the folded
Donaldson--Thomas transformation behind the degree pairing, and
Appendix~\ref{app:other-models} records two further cluster models without
proofs.  The ancillary file \path{Kron322.py} checks the finite matrix, fan,
semigroup, Donaldson--Thomas, and summation calculations and compares the
closed formula with the symmetric-group character formula for all
admissible triples of size at most twelve.

Throughout, $\kk$ is an algebraically closed field of characteristic zero.
Partitions are padded by trailing zeros when necessary.  We write
$110=(1,1,0)$, $21=(2,1)$, and so forth.

\section{Tensor invariants and Kronecker coefficients}
\label{sec:invariants}

\subsection{Tensor \texorpdfstring{$U$}{U}-invariant algebras}

Let $\lambda\vdash N$ be a partition.  We write $\Specht^\lambda$ for
the corresponding Specht module of the symmetric group $\mathfrak S_N$,
and $\Schur_\lambda V$ for the corresponding Schur module.  For three
partitions of $N$, the Kronecker coefficient is defined by
\[ \Specht^\mu\otimes\Specht^\nu \cong\bigoplus_{\lambda\vdash N} g_{\lambda,\mu,\nu}\Specht^\lambda. \]

For positive integers $a,b,c$, set
\[ T_{a,b,c}=\kk^a\otimes\kk^b\otimes\kk^c. \]
Throughout, $\kk[T_{a,b,c}]$ denotes the polynomial coordinate ring
$\Sym(T_{a,b,c}^{\vee})$, with the three general linear groups acting on
functions contragrediently.  We use the standard polynomial highest-weight
labels in the resulting Cauchy decomposition and suppress dual symbols on
Schur functors.  Let $U_d\subset\GL_d$ be the standard upper unitriangular
subgroup.
The double- and triple-$U$-invariant algebras are
\[ \SR_{a,b,c} :=\kk[T_{a,b,c}]^{U_a\times U_b}, \qquad \UR_{a,b,c} :=\kk[T_{a,b,c}]^{U_a\times U_b\times U_c}. \]
They are graded by triples of dominant weights.  When
$(a,b,c)=(3,2,2)$, we abbreviate them to $\SR$ and $\UR$.  In this
case the last factor $U_2$ is isomorphic to $\mathbb G_a$.  Set
\[ X_{\mathrm{dbl}}:=\Spec\SR, \qquad Y_{322}:=\Spec\UR, \qquad \pi:X_{\mathrm{dbl}}\longrightarrow Y_{322}. \]
Then $Y_{322}=X_{\mathrm{dbl}}/\!/\mathbb G_a$ is the affine categorical
quotient: every $\mathbb G_a$-invariant morphism from
$X_{\mathrm{dbl}}$ to an affine scheme factors uniquely through $\pi$.
All subsequent references to the quotient mean $Y_{322}$.  The function
field calculation in \cref{prop:birational-quotient-chart} gives dimension
seven, and \cref{thm:thirteen-generate} proves finite generation; hence
$Y_{322}$ is an affine variety.  Before those results, the same notation is
understood scheme-theoretically.  The standard Cauchy--Schur--Weyl
realization used below is recalled, for example, in
\cite[Chapter~I]{Macdonald}.

\begin{lemma}
\label{lem:kronecker-weight-space}
For partitions $\lambda,\mu,\nu$ of the same size satisfying
$\ell(\lambda)\le a$, $\ell(\mu)\le b$, and $\ell(\nu)\le c$, one has
\[ \dim (\UR_{a,b,c})_{\lambda,\mu,\nu} =g_{\lambda,\mu,\nu}. \]
\end{lemma}

\begin{proof}
The Cauchy decomposition and the restriction of a Schur functor to a tensor
product give
\[
 \Sym^N(\kk^a\otimes\kk^b\otimes\kk^c)
 \cong
 \bigoplus_{\lambda,\mu,\nu\vdash N}
 g_{\lambda,\mu,\nu}
 \bigl(\Schur_\lambda\kk^a\otimes
       \Schur_\mu\kk^b\otimes
       \Schur_\nu\kk^c\bigr).
\]
Each irreducible polynomial $\GL$-module has a one-dimensional highest
weight line.  Taking invariants under the three maximal unipotent subgroups
therefore retains precisely $g_{\lambda,\mu,\nu}$ copies of the indicated
triple highest-weight line.
\end{proof}

\subsection{Determinant reduction and stabilization}

Fix $b,c\ge1$, put $D=bc$, and assume $D\ge2$.  Let
$A=\kk^D$, $B=\kk^b$, $C=\kk^c$, and $E=B\otimes C$.  A tensor in
$A\otimes B\otimes C=A\otimes E$ is a square $D\times D$ matrix.  Let
\[ \Delta_{b,c} =\det\bigl(\operatorname{Flat}_{A\mid B\otimes C}\bigr) \in\Sym^D(A\otimes B\otimes C) \]
be the determinant of this flattening.

\begin{proposition}
\label{prop:det-extension}
Under the standard inclusion $\kk^{D-1}\subset\kk^D$, one has a direct
sum decomposition
\[ \UR_{D,b,c} =\bigoplus_{r\ge0}\Delta_{b,c}^{\,r}\UR_{D-1,b,c}. \]
Consequently,
\[ \UR_{bc,b,c}\cong\UR_{bc-1,b,c}\otimes_\kk\kk[\Delta_{b,c}] \]
as multigraded algebras, and
\[ \deg\Delta_{b,c} =\bigl((1^{bc}),(c^b),(b^c)\bigr). \]
\end{proposition}

\begin{proof}
The Cauchy decomposition is
\[ \Sym(A\otimes E) =\bigoplus_{\ell(\lambda)\le D} \Schur_\lambda A\otimes\Schur_\lambda E. \]
If $\lambda_D=0$, the $U_D$-highest-weight line in
$\Schur_\lambda A$ is the image of the $U_{D-1}$-highest-weight line in
$\Schur_\lambda\kk^{D-1}$.  Hence the sum of the components with
$\lambda_D=0$ is naturally $\UR_{D-1,b,c}$.

The element $\Delta_{b,c}$ spans $\det(A)\otimes\det(E)$, and
\[ \det(E)\cong(\det B)^c\otimes(\det C)^b. \]
For $\lambda=(\lambda_1,\ldots,\lambda_D)$, put $r=\lambda_D$ and
$\bar\lambda=\lambda-r(1^D)$.  The determinant-twist identities
\[ \Schur_\lambda A\cong(\det A)^r\otimes\Schur_{\bar\lambda}A, \qquad \Schur_\lambda E\cong(\det E)^r\otimes\Schur_{\bar\lambda}E \]
are realized by multiplication by $\Delta_{b,c}^{\,r}$.  After
restriction from $\GL(E)$ to $\GL(B)\times\GL(C)$, the last factor shifts
the other two highest weights by $rc(1^b)$ and $rb(1^c)$, respectively.
Thus every homogeneous component is uniquely a power of $\Delta_{b,c}$
times a component with last part zero.  The sum is direct because the power
is recovered from $\lambda_D$.
\end{proof}

\begin{corollary}
\label{cor:rectangular-translation}
Let $r=\lambda_{bc}$.  If $g_{\lambda,\mu,\nu}\ne0$, then
\[ \mu_b\ge rc,\qquad \nu_c\ge rb, \]
and
\[ g_{\lambda,\mu,\nu} =g_{\lambda-r(1^{bc}),\,\mu-rc(1^b),\,\nu-rb(1^c)}. \]
If either inequality fails, then $g_{\lambda,\mu,\nu}=0$.
\end{corollary}

\begin{proof}
Take dimensions in the homogeneous-component isomorphisms in the proof of
\cref{prop:det-extension} and apply \cref{lem:kronecker-weight-space}.
\end{proof}

\begin{corollary}
\label{cor:stable-first-factor}
For every $a\ge bc$, the standard inclusion $\kk^{bc}\subset\kk^a$
induces an isomorphism
\[ \UR_{bc,b,c}\xrightarrow{\sim}\UR_{a,b,c}, \]
where the first partition is padded by $a-bc$ zeros.  Hence
\[ \UR_{a,b,c}\cong\UR_{bc-1,b,c}\otimes_\kk\kk[x] \qquad(a\ge bc). \]
For $a>bc$, the new generator is the highest-weight $bc\times bc$
maximal minor of the $a\times bc$ flattening.
\end{corollary}

\begin{proof}
Only partitions of length at most $bc=\dim E$ occur in the Cauchy
decomposition of $\Sym(\kk^a\otimes E)$.  Their $U_a$-highest-weight
lines are already contained in the standard copy of $\kk^{bc}$.  The first
assertion follows componentwise, and the second follows from
\cref{prop:det-extension}.
\end{proof}

\begin{remark}
At the level of Kronecker coefficients, \cref{cor:rectangular-translation}
is the rectangular translation symmetry of Briand--Orellana--Rosas
\cite[Proposition~7]{BOR}; the $(4,2,2)$ case also appears in
\cite[Corollary~5.1]{Vallejo}.  The proposition above is the compatible
algebra-level statement.
\end{remark}

\section{Cluster, upper, and middle algebras}
\label{sec:cluster-background}

\subsection{Cluster and upper cluster algebras}

Let $\mu,f\ge0$, put $m=\mu+f$, let
$\cF=\kk(x_1,\ldots,x_m)$, and let
$B=(b_{ki})\in M_{\mu\times(\mu+f)}(\ZZ)$ have skew-symmetrizable
principal part.  Throughout the paper, rows are indexed by the $\mu$
mutable variables and columns by all $\mu+f$ mutable and frozen variables.
A seed is a pair $\Sigma=(\mathbf{x},B)$, where
$\mathbf{x}=(x_1,\ldots,x_m)$ is a transcendence basis of $\cF$.  The
first $\mu$ variables are mutable and the remaining $f$ variables are
frozen.  Mutation in direction $k\le\mu$ replaces $x_k$ by $x_k'$, where
\[ x_kx_k' =\prod_{b_{ki}>0}x_i^{b_{ki}} +\prod_{b_{ki}<0}x_i^{-b_{ki}}, \]
and mutates $B$ by the usual matrix rule in this convention.  We write
\[ \cL(\mathbf{x}) =\kk[x_1^{\pm1},\ldots,x_\mu^{\pm1},x_{\mu+1},\ldots,x_{\mu+f}]. \]
The cluster algebra is generated by all cluster variables in the mutation
class, whereas the upper cluster algebra is
\[ \cU(\Sigma)=\bigcap_{\Sigma'\sim\Sigma}\cL(\mathbf{x}'). \]
The Laurent phenomenon gives $\cA(\Sigma)\subseteq\cU(\Sigma)$; our
conventions follow \cite{FZ1,BFZ}.

\subsection{Theta functions and the middle algebra}

For a cluster seed, Gross--Hacking--Keel--Kontsevich construct theta functions
from the scattering diagram.  In initial seed coordinates they have the form
\[ \theta_g=\mathbf{x}^gF_g(\mathbf{y}), \qquad F_g\in1+\ZZ_{\ge0}\llbracket\mathbf{y}\rrbracket. \]
Let $D_f$ be the boundary divisor associated with a frozen index $f$,
and let $\vartheta_f^\vee$ be the corresponding theta function on the dual
cluster variety.  With the min-plus convention, the cone is
\[ \Xi =\left\{g\in\RR^m\ \middle|\ (\vartheta_f^\vee)^{\mathrm{trop}}(g)\ge0 \text{ for every frozen }f\right\}. \]
If
\[ \vartheta_f^\vee=\sum_{\alpha\in A_f}c_\alpha\mathbf{x}^\alpha, \qquad c_\alpha>0, \]
then
\[ (\vartheta_f^\vee)^{\mathrm{trop}}(g) =\min_{\alpha\in A_f}\langle\alpha,g\rangle, \]
so $\Xi$ is a rational polyhedral cone.  The middle algebra is the theta
subalgebra regular along the chosen boundary.  Under the hypotheses of
\cite{GHKK}, and in the valuative formulation of \cite{CMMM}, it has basis
\begin{equation}
\label{eq:middle-theta-basis}
 \{\theta_g\mid g\in\Xi\cap\ZZ^m\}.
\end{equation}
Valuative independence implies that regularity may be tested one theta
function at a time.

Fix a translation-invariant total order on $\ZZ^m$.  In a seed torus, the
seed valuation of a nonzero Laurent polynomial is the least exponent in its
support.  This lowest-term valuation has one-dimensional leaves: if two
Laurent polynomials have the same least exponent, subtracting the ratio of
their leading coefficients strictly increases the valuation.  This is the
standard lowest-term example in \cite[Example~2.2(2)]{KM}.

Under the following hypotheses, a Hilbert basis of the value semigroup
generates the algebra.

\begin{lemma}
\label{lem:hilbert-theta-generation}
Suppose the finite theta functions in \eqref{eq:middle-theta-basis} are
homogeneous for a grading with finite-dimensional homogeneous pieces, and
admit a seed valuation with one-dimensional leaves and
$\initer(\theta_g)=\mathbf{x}^g$.  If $\mathcal H$ is the Hilbert basis
of $\Xi\cap\ZZ^m$, then the theta functions
$\{\theta_h\mid h\in\mathcal H\}$ generate the middle algebra.
\end{lemma}

\begin{proof}
Let $B$ be the algebra generated by the $\theta_h$, $h\in\mathcal H$.
For $g\in\Xi\cap\ZZ^m$, choose a decomposition
$g=h_1+\cdots+h_r$ with $h_i\in\mathcal H$.  After rescaling, the product
$\theta_{h_1}\cdots\theta_{h_r}$ and $\theta_g$ have the same initial
term $\mathbf{x}^g$.  Their difference is homogeneous of the same weight
and has strictly larger valuation.  The one-dimensional-leaf property
implies that only finitely many valuation levels occur in a fixed
finite-dimensional homogeneous component.  Repeating the subtraction
therefore terminates and expresses $\theta_g$ as an element of $B$.
Since the $\theta_g$ form a basis, $B$ is the middle algebra.
\end{proof}

\section{The affine quotient and the coefficient family}
\label{sec:slice-volume}

\subsection{Logarithmic charts on the affine quotient}

An algebraic torus with coordinates $z_1,\ldots,z_d$ carries the
logarithmic form
\[ \Omega_{\mathbf z}=\bigwedge_{i=1}^d d\!\log z_i =\frac{dz_1\wedge\cdots\wedge dz_d}{z_1\cdots z_d}. \]
It has neither zeros nor poles on the torus, and cluster mutations preserve
it up to sign.  In a geometric cluster $\mathcal A$-realization, the frozen
coordinates are detected by the nonzero residues along their zero divisors
\cite[Definition~11.2, Lemma~11.3, and Definition~11.5]{GS15}.  Cluster
varieties are obtained by gluing such tori by birational maps preserving the
logarithmic form \cite[\S1]{GHK15}.  Non-invertible frozen variables and the corresponding
boundary divisors in partial compactifications are treated in
\cite[\S2.2]{MandelCox}.

Contract the logarithmic top form of the seed on $\SR$ with the vector
field of the last root subgroup and restrict the contraction to a local
slice.  The resulting horizontal invariant form descends to a rational top
form $\omega_{322}$ on $Y_{322}$.  Compare \cite[\S1.4]{GS15} for the
analogous contraction along a freely acting torus.  For a seven-function
chart $\mathbf z$ with frozen product $P_F$, the required identity is
\[ P_F\bigwedge_{i=1}^7 d\!\log z_i=c\,\omega_{322} \qquad(c\in\kk^\times). \]
If the quotient of the two sides is a nonzero constant, then the seven
functions are algebraically independent and the frozen divisor has the
prescribed multiplicities.  For two charts with the same frozen product, the identity fixes the sign of
the transition.  The coefficient $\zeta$ then places the three exchanges
in an ordinary cluster family.

\subsection{The quotient slice and its volume form}

The cluster structure on the double-$U$-invariant algebra $\SR$ is recalled
in Appendix~\ref{app:partial-ring}.  Write its initial extended cluster as
\begin{equation}
\label{eq:double-initial-cluster}
 \mathbf s^\circ=(s_1,s_2,s_3\,;s_4,\ldots,s_8),
\end{equation}
with the first three variables mutable.  Let $\partial$ be the locally
nilpotent derivation induced by the positive root subgroup of the last
$\GL_2$.  The root strings needed below are
\[ \partial(s_2)=s_3,\qquad \partial(s_4)=s_9,\qquad \partial(s_9)=2s_5,\qquad \partial(s_7)=-s_8, \]
while $\partial$ kills $s_1,s_3,s_5,s_6,s_8$.  Here
\begin{equation}
\label{eq:middle-root-coefficient}
 s_9=s_{110,11}^{11}
 =\frac{s_1s_6+s_2^2s_5+s_3^2s_4}{s_2s_3}.
\end{equation}
The signs are fixed by $(X,Y)\mapsto(X,Y+tX)$ on the last tensor factor, and
\[ \UR=\ker(\partial:\SR\longrightarrow\SR). \]

On $D(s_3)$ set $\xi=s_2/s_3$.  Then $\partial(\xi)=1$, and the local-slice
theorem gives
\[ \SR[s_3^{-1}]=\UR[u_1^{-1}][\xi] \qquad (u_1=s_3) \]
\cite[Chapter~1]{Freudenburg}.  The following invariants will be used as
coordinates on the quotient:
\begin{equation}
\label{eq:quotient-coordinates}
 \begin{aligned}
 q_1&=u_3=s_1,&
 q_2&=u_1=s_3,&
 q_3&=u_2=s_5,\\
 q_4&=u_4=s_6,&
 q_5&=u_5=s_8,&
 q_6&=u_6=s_3s_9-2s_2s_5,\\
 &&q_7&=u_8=s_2s_8+s_3s_7.
 \end{aligned}
\end{equation}

\begin{proposition}
\label{prop:birational-quotient-chart}
The functions $\xi,q_1,\ldots,q_7$ are birational coordinates.  More
precisely,
\[
\begin{gathered}
 s_1=q_1,\quad s_2=\xi q_2,\quad s_3=q_2,\quad
 s_5=q_3,\quad s_6=q_4,\quad s_8=q_5,\\
 s_7=\frac{q_7}{q_2}-\xi q_5,
 \qquad
 s_4=\xi^2q_3+\frac{\xi q_6}{q_2}-\frac{q_1q_4}{q_2^2}.
 \end{gathered}
\]
Consequently,
\[ \kk(X_{\mathrm{dbl}})=\kk(q_1,\ldots,q_7)(\xi), \qquad \kk(Y_{322})=\kk(q_1,\ldots,q_7). \]
\end{proposition}

\begin{proof}
The inverse formulas follow successively from
\eqref{eq:quotient-coordinates} and \eqref{eq:middle-root-coefficient}.
Since $\partial(\xi)=1$ and $\partial(q_i)=0$, the rational kernel of
$\partial$ is $\kk(q_1,\ldots,q_7)$.
\end{proof}

Normalize the logarithmic form of the upstairs seed by
\[ \omega_s =s_4s_5s_6s_7s_8\bigwedge_{i=1}^8\dlog s_i =\frac{ds_1\wedge\cdots\wedge ds_8}{s_1s_2s_3}. \]

\begin{proposition}
\label{prop:quotient-volume}
The contraction $\omega_{322}:=\iota_\partial\omega_s$ descends to
$Y_{322}$ and
\begin{equation}
\label{eq:quotient-volume}
 \omega_{322}=-\frac{dq_1\wedge\cdots\wedge dq_7}{q_1q_2^3}.
\end{equation}
\end{proposition}

\begin{proof}
The divergence of $\partial$ with respect to
$ds_1\wedge\cdots\wedge ds_8$ is $s_3/s_2$, while
$\partial\log((s_1s_2s_3)^{-1})=-s_3/s_2$.  Hence
$\mathcal L_\partial\omega_s=0$, so the contraction is basic.  Finally,
\[ \det\frac{\partial(\xi,q_1,\ldots,q_7)} {\partial(s_1,\ldots,s_8)}=-\frac{s_3^2}{s_2}. \]
Using $\omega_s=d\xi\wedge\omega_{322}$ and
$s_1=q_1,s_3=q_2$ gives \eqref{eq:quotient-volume}.
\end{proof}

\subsection{The four-frozen logarithmic seed}

For a seven-function rational chart $\mathbf z=(z_1,\ldots,z_7)$ and a
chosen frozen subset $F$, put $P_F=\prod_{f\in F}f$ and
\[ \mathcal V_{\mathbf z,F}(q) =-q_1q_2^3P_F(q) \det\left(\frac{\partial\log z_i}{\partial q_j}\right)_{1\le i,j\le7}. \]
Then
\[ P_F\bigwedge_{i=1}^7\dlog z_i =\mathcal V_{\mathbf z,F}(q)\,\omega_{322}. \]
A logarithmic chart must have $\mathcal V_{\mathbf z,F}\in\kk^\times$.
Since each $d\!\log z_i$ has weight zero, its frozen product has the same
character as $\omega_{322}$, namely
\[ \kappa_{\mathrm{ac}}:=\wt(\omega_{322})=(6,4,2;7,5;7,5). \]

We use the following thirteen normalized highest-weight invariants:
\begin{equation}
\label{eq:thirteen-invariant-weights}
\begin{array}{llll}
 u_1=u_{100,10}^{10},&u_2=u_{110,11}^{20},&u_3=u_{110,20}^{11},&u_4=u_{200,11}^{11},\\
 u_5=u_{111,21}^{21},&u_6=u_{210,21}^{21},&u_7=u_{211,22}^{31},&u_8=u_{211,31}^{22},\\
 u_9=u_{220,22}^{22},&u_{10}=u_{221,32}^{32},&u_{11}=u_{321,33}^{42},&u_{12}=u_{321,42}^{33},\\
 u_{13}=u_{222,33}^{33}.&&&
\end{array}
\end{equation}
The coordinates in \eqref{eq:quotient-coordinates} fix the normalizations
of $u_1,\ldots,u_6,u_8$; the identities below and
\eqref{eq:basic-relations} fix the rest.  Since
$\wt(u_2u_3u_4u_{13})=\kappa_{\mathrm{ac}}$, take
$u_1,u_5,u_7$ mutable and
$F_4=\{u_2,u_3,u_4,u_{13}\}$ frozen.
The two identities needed for the adjacent charts are
\begin{equation}
\label{eq:two-chart-relations}
 u_3u_7=u_2u_8+u_5u_6,
 \qquad
 u_7u_8=u_4u_5^2+u_1^2u_{13}.
\end{equation}
They are direct identities of the normalized invariants; in the
$q$-coordinates they become the two equations used below.

\begin{proposition}
\label{prop:volume-selected-seed}
The seven-tuples
\[ \Sigma_4=(u_1,u_5,u_7\,;u_2,u_3,u_4,u_{13}), \qquad \Sigma_4'=(u_1,u_5,u_8\,;u_2,u_3,u_4,u_{13}) \]
are birational logarithmic charts and satisfy (in the above order)
\begin{equation}
\label{eq:two-volume-identities}
 \begin{aligned}
 u_2u_3u_4u_{13}\bigwedge_{z\in\Sigma_4}\dlog z&=-\omega_{322},\qquad
  u_2u_3u_4u_{13}\bigwedge_{z\in\Sigma_4'}\dlog z&=\omega_{322}.
 \end{aligned}
\end{equation}
If $U_4,U_4'$ are their smooth partial-chart loci and
$U=U_4\cup U_4'$, then $K_U+D_F|_U\sim0$.
\end{proposition}

\begin{proof}
Put $r=u_7$ and $s=u_{13}$.  In the $q$-coordinates,
\eqref{eq:two-chart-relations} reads
\[ q_1r=q_3q_7+q_5q_6, \qquad q_2^2s=rq_7-q_4q_5^2. \]
Modulo $dq_1,\ldots,dq_5$ this gives
\[ dr\wedge ds\equiv \frac{rq_5}{q_1q_2^2}\,dq_6\wedge dq_7. \]
Substitution in the displayed order of $\Sigma_4$, followed by
\eqref{eq:quotient-volume}, gives the first identity in
\eqref{eq:two-volume-identities}.

For the second, write $F=u_4u_5^2+u_1^2u_{13}$.  Since $u_7u_8=F$,
\[ \dlog u_8=-\dlog u_7+\dlog F. \]
The last term lies in the span of
$\dlog u_1,\dlog u_5,\dlog u_4,\dlog u_{13}$ and therefore vanishes in the
top wedge.  Replacing $u_7$ by $u_8$ changes only the sign, proving the
second identity.  The two relations recover the missing quotient coordinates from either
seven-tuple, so both charts are birational.

On $U_4$ the logarithmic form has divisor $-D_F|_{U_4}$, and similarly on
$U_4'$.  By \eqref{eq:two-volume-identities} the two forms differ by $-1$
on the overlap; they therefore glue to a nowhere-vanishing section of
$\omega_U(D_F|_U)$.  Hence $K_U+D_F|_U\sim0$.
\end{proof}

The remaining quotient exchanges are
\[ u_1u_1'=u_2u_4u_5^2+u_3u_7^2, \qquad u_5u_5'=u_1^2u_2u_{13}-u_3u_7^2. \]
To absorb the sign in the second exchange, introduce an invertible
coefficient $\zeta$.  Let $\Sigma_{u,\zeta}$ have mutable variables
$(u_1,u_5,u_7)$ and frozen coefficients
$(u_2,u_3,u_4,u_{13},\zeta)$ over
$K_\zeta=\kk[u_2,u_3,u_4,u_{13},\zeta^{\pm1}]$.  On the fiber
$\zeta=-1$ we use the order
\begin{equation}
\label{eq:markov-variable-order}
 \mathbf u^\circ=(u_1,u_5,u_7\,;u_2,u_3,u_4,u_{13}).
\end{equation}
In the family order
$(u_1,u_5,u_7\,;u_2,u_3,u_4,u_{13},\zeta)$, the mutable-row exchange matrix is
\[
 B_{u,\zeta}=
 \begin{bmatrix}
  0&-2& 2&-1& 1&-1& 0&0\\
  2& 0&-2& 1&-1& 0& 1&1\\
 -2& 2& 0& 0& 0& 1&-1&0
 \end{bmatrix}.
\]
Its exchanges are
\begin{equation}
\label{eq:zeta-family-exchanges}
 \begin{aligned}
 u_1u_{1,\zeta}'&=u_2u_4u_5^2+u_3u_7^2,\\
 u_5u_{5,\zeta}'&=u_3u_7^2+\zeta u_1^2u_2u_{13},\\
 u_7u_{7,\zeta}'&=u_4u_5^2+u_1^2u_{13}.
 \end{aligned}
\end{equation}
At $\zeta=-1$, the first and third exchanges are unchanged and
$u_{5,\zeta}'=-u_5'$.  Thus the signed quotient chart is the negative fiber
of an ordinary cluster family.

Appendix~\ref{app:D5-parabolic} gives a $D_5$-parabolic and integral
Gale-dual interpretation of this seed.

\section{The invariant ring, the four-chart upper bound, and the middle algebra}
\label{sec:ring-equalities}

\subsection{The invariant ring and the four-chart upper bound}
\label{subsec:invariant-upper}

Put
\[ \Ugen=\kk[u_1,\ldots,u_{13}]\subseteq\UR. \]
The generators satisfy the six identities
\begin{equation}
\label{eq:basic-relations}
 \begin{alignedat}{2}
 u_3u_7&=u_2u_8+u_5u_6,
 &\qquad u_1^2u_9&=u_6^2+4u_2u_3u_4,\\
 u_1u_{10}&=2u_2u_8+u_5u_6,
 &u_1u_{11}&=-2u_2u_4u_5-u_6u_7,\\
 u_1u_{12}&=2u_3u_4u_5-u_6u_8,
 &u_1^2u_{13}&=u_7u_8-u_4u_5^2.
 \end{alignedat}
\end{equation}
They follow by substitution in the quotient coordinates of
\cref{prop:birational-quotient-chart}.  Choose the generating set $s_1,\ldots,s_{14}$ of $\SR$ recorded in
\cref{tab:partial-generators}; it is obtained from the minimal generating
set in Appendix~\ref{app:partial-ring} by adjoining one redundant one-step
cluster variable.  The nonzero values of $\partial$ are
\begin{equation}
\label{eq:partial-on-s-generators}
 \begin{gathered}
 \partial(s_2)=s_3,\qquad \partial(s_4)=s_9,\qquad
 \partial(s_7)=-s_8,\qquad \partial(s_9)=2s_5,\\
 \partial(s_{10})=s_{11},\qquad
 \partial(s_{12})=2s_{10},\qquad
 \partial(s_{13})=2s_5s_2.
 \end{gathered}
\end{equation}
The remaining seven generators are killed by $\partial$.  In terms of
this set, the thirteen invariants are
\begin{equation}
\label{eq:u-in-terms-of-s}
 \begin{array}{lll}
 u_1=s_3,&u_2=s_5,&u_3=s_1,\\
 u_4=s_6,&u_5=s_8,&u_7=s_{11},\\
 u_{13}=s_{14},&u_6=s_3s_9-2s_5s_2,&u_8=s_3s_7+s_2s_8,\\
 u_9=s_9^2-4s_5s_4,&u_{10}=s_9s_8+2s_5s_7,&u_{11}=2s_5s_{10}-s_9s_{11},\\
 \multicolumn{3}{l}{u_{12}=s_2s_9s_8-s_3s_9s_7+2s_2s_5s_7-2s_3s_4s_8.}
 \end{array}
\end{equation}
The root strings \eqref{eq:partial-on-s-generators} show directly that all
expressions in \eqref{eq:u-in-terms-of-s} lie in $\ker\partial$.

Consider the two local slices
\[ \xi_1=\frac{s_2}{s_3}=\frac{s_2}{u_1}, \qquad \xi_2=\frac{s_9}{2s_5}=\frac{s_9}{2u_2}. \]
Both satisfy $\partial(\xi_i)=1$.  They arise from root strings of lengths two and three; their denominators
$u_1$ and $u_2$ are nonassociate invariants of degrees one and two.  For $i=1,2$, set
\[ \rho_i(g)=\exp(-\xi_i\partial)g =\sum_{j\ge0}\frac{(-\xi_i)^j}{j!}\partial^j(g). \]
The sum is finite because $\partial$ is locally nilpotent.  The
nonzero values of the two local-slice retractions are
\begin{equation}
\label{eq:slice-retraction-table}
\begin{array}{c|c|c}
 g&\rho_1(g)&\rho_2(g)\\ \hline
 s_{10}&-u_4u_5/u_1&u_{11}/(2u_2)\\
 s_7&u_8/u_1&u_{10}/(2u_2)\\
 s_{12}&u_4u_8/u_1^2&u_7u_9/(4u_2^2)\\
 s_{13}&u_3u_4/u_1&u_1u_9/(4u_2)\\
 s_2&0&-u_6/(2u_2)\\
 s_4&-u_3u_4/u_1^2&-u_9/(4u_2)\\
 s_9&u_6/u_1&0
\end{array}
\end{equation}
The generators $s_1,s_3,s_5,s_6,s_8,s_{11},s_{14}$ are already
invariant.  The local-slice theorem \cite[Chapter~1]{Freudenburg} and \eqref{eq:slice-retraction-table}
therefore give
\begin{equation}
\label{eq:two-localized-equalities}
 \UR[u_1^{-1}]=\Ugen[u_1^{-1}],
 \qquad
 \UR[u_2^{-1}]=\Ugen[u_2^{-1}].
\end{equation}

\begin{lemma}[Intersection of two localizations]
\label{lem:two-localization}
Let $A$ be a domain and let $f,g\in A$ be nonzero.  If multiplication
by $g$ is injective on $A/fA$, then, inside $\Frac(A)$,
\[ A[f^{-1}]\cap A[g^{-1}]=A. \]
\end{lemma}

\begin{proof}
Write an element of the intersection as
$x=a/f^p=b/g^q$.  Then $g^qa=f^pb$.  Modulo $f$, injectivity of
multiplication by $g^q$ gives $a\in fA$.  Cancelling one power of
$f$ and iterating removes the entire denominator $f^p$, so $x\in A$.
\end{proof}

\begin{lemma}
\label{lem:Ugen-regular-sequence}
The pair $u_1,u_2$ is a regular sequence in $\Ugen$.  Equivalently,
multiplication by $u_2$ is injective on $\Ugen/u_1\Ugen$.
\end{lemma}

\begin{proof}
Since $\Ugen$ is a domain, $u_1$ is a non-zero-divisor.  The colon
calculation in Appendix~\ref{app:regular-sequence} gives
\[ (I+(U_1)):U_2=I+(U_1) \]
for the defining ideal $I$ of $\Ugen$, which is the asserted
injectivity modulo $u_1$.
\end{proof}

\begin{theorem}
\label{thm:thirteen-generate}
The thirteen functions in \eqref{eq:thirteen-invariant-weights} generate the triple-$U$-invariant algebra:
\[ \UR=\Ugen. \]
\end{theorem}

\begin{proof}
By \cref{lem:two-localization,lem:Ugen-regular-sequence},
\begin{equation}
\label{eq:Ugen-localization-intersection}
 \Ugen[u_1^{-1}]\cap\Ugen[u_2^{-1}]=\Ugen.
\end{equation}
Let $f\in\UR$.  The two slice equalities
\eqref{eq:two-localized-equalities} place $f$ in both rings on the
left-hand side of \eqref{eq:Ugen-localization-intersection}; hence
$f\in\Ugen$.  The reverse inclusion follows from the definition of $\Ugen$.
\end{proof}

\begin{lemma}
\label{lem:relevant-one-dimensional-weights}
For each weight in the following table, the monomial is the unique
monomial in $u_1,\ldots,u_{13}$ of that weight:
\begin{equation}
\label{eq:unique-weight-monomials}
\begin{array}{c|c@{\qquad}c|c}
\text{weight}&\text{unique monomial}&\text{weight}&\text{unique monomial}\\ \hline
(210;21;21)&u_6&(221;32;32)&u_{10}\\
(321;33;42)&u_{11}&(321;42;33)&u_{12}\\
(320;32;32)&u_9u_1&(331;43;43)&u_9u_5\\
(431;44;53)&u_9u_7&(431;53;44)&u_9u_8
\end{array}
\end{equation}
Consequently, the corresponding homogeneous components of $\UR$ are
one-dimensional.
\end{lemma}

\begin{proof}
By \cref{thm:thirteen-generate}, every homogeneous element of $\UR$ is a
linear combination of monomials in the thirteen generators.  For each entry in \eqref{eq:unique-weight-monomials}, the seven
nonnegative integer equations determined by
\eqref{eq:thirteen-invariant-weights} have the exponent vector as
their unique solution.  Hence the corresponding component is spanned by the monomial.  
The \texttt{weights} task in \path{Kron322.py} checks
these eight semigroup searches.
\end{proof}

Put
\[ K_u=\kk[u_2,u_3,u_4,u_{13}]. \]
Specializing the one-step variables of \eqref{eq:zeta-family-exchanges}, set
\[ u_1'=u_{1,\zeta}'\big|_{\zeta=-1},\qquad u_5'=-u_{5,\zeta}'\big|_{\zeta=-1},\qquad u_7'=u_{7,\zeta}'\big|_{\zeta=-1}. \]
They are regular invariants; the basic relations give
\begin{equation}
\label{eq:partners-in-u}
 u_1'=u_7u_{10}-u_1u_2u_{13},
 \qquad
 u_5'=u_1u_{11}+u_2u_4u_5,
 \qquad
 u_7'=u_8.
\end{equation}

The Laurent rings of the initial seed $\Sigma_{u,\zeta}$ and its three
one-step mutations are denoted by
$\cL_\circ(\zeta),\cL_1(\zeta),\cL_5(\zeta),\cL_7(\zeta)$:
\begin{equation*}
 \begin{aligned}
 \cL_\circ(\zeta)&=K_\zeta[u_1^{\pm1},u_5^{\pm1},u_7^{\pm1}],\\
 \cL_1(\zeta)&=K_\zeta[(u_{1,\zeta}')^{\pm1},u_5^{\pm1},u_7^{\pm1}],\\
 \cL_5(\zeta)&=K_\zeta[u_1^{\pm1},(u_{5,\zeta}')^{\pm1},u_7^{\pm1}],\\
 \cL_7(\zeta)&=K_\zeta[u_1^{\pm1},u_5^{\pm1},(u_{7,\zeta}')^{\pm1}].
 \end{aligned}
\end{equation*}
Write
\[ \cL_t^-:=\cL_t(\zeta)/(\zeta+1)\cL_t(\zeta) \]
for the specialization at $\zeta=-1$.  Since $-1$ is a unit,
these four rings are identified with
\begin{equation*}
 \begin{aligned}
 \cL_\circ^-&=K_u[u_1^{\pm1},u_5^{\pm1},u_7^{\pm1}],\\
 \cL_1^-&=K_u[(u_1')^{\pm1},u_5^{\pm1},u_7^{\pm1}],\\
 \cL_5^-&=K_u[u_1^{\pm1},(u_5')^{\pm1},u_7^{\pm1}],\\
 \cL_7^-&=K_u[u_1^{\pm1},u_5^{\pm1},(u_7')^{\pm1}].
 \end{aligned}
\end{equation*}
The \texttt{laurent} task in \path{Kron322.py} writes
explicit Laurent expressions for all thirteen $u_i$ in these four charts
and verifies the six relations in \eqref{eq:basic-relations}.
Define the \emph{four-chart upper bound at $\zeta=-1$} by
\[ \Uneg:=\cL_\circ^-\cap\cL_1^-\cap\cL_5^-\cap\cL_7^- \subseteq\Frac(\UR). \]
The four rings are the specializations of the initial chart and its three
one-step mutations.  No commutation of specialization with the infinite
intersection defining $\cU(\Sigma_{u,\zeta})$ is used.

\begin{lemma}
\label{lem:four-chart-coprimality}
The elements $u_1,u_5,u_7$ are pairwise nonassociate primes in the UFD
$\UR$, and
\[ \gcd(u_1,u_1')=\gcd(u_5,u_5')=\gcd(u_7,u_7')=1. \]
\end{lemma}

\begin{proof}
The algebra $\UR$ is factorial by the factoriality theorem for invariants
of a factorial affine variety under a connected group with trivial character
group, applied to $U_3\times U_2\times U_2$; see \cite{PV}.  Its only
units are scalars, by the positive polynomial-degree grading.

Because the grading torus is connected, every irreducible factor of a
homogeneous semi-invariant in the UFD $\UR$ is again a homogeneous
semi-invariant.  The only semigroup weight of polynomial degree one is
$\wt(u_1)$.  Hence a factorization of $u_5$ would have degree split
$1+2$, but
\[ \wt(u_5)-\wt(u_1)=(011;11;11) \]
is not dominant.  A factorization of $u_7$ has split $1+3$ or
$2+2$.  The first is excluded by
\[ \wt(u_7)-\wt(u_1)=(111;12;21). \]
The degree-two semigroup weights and the corresponding differences are
\[
\begin{array}{c|c}
\text{degree-two weight }\eta&\wt(u_7)-\eta\\ \hline
2\wt(u_1)&(011;02;11)\\
\wt(u_2)&(101;11;11)\\
\wt(u_3)&(101;02;20)\\
\wt(u_4)&(011;11;20).
\end{array}
\]
Every difference is nondominant.  Thus $u_1,u_5,u_7$ are irreducible,
hence prime, and their distinct weights make them pairwise nonassociate.

Modulo these primes, \eqref{eq:partners-in-u} gives
\[ u_1'\equiv u_7u_{10}\pmod{u_1},\qquad u_5'\equiv u_1u_{11}\pmod{u_5},\qquad u_7'=u_8. \]
The weight differences are
\[
\begin{array}{c|c}
\text{putative quotient}&\text{weight difference}\\ \hline
u_7/u_1&(111;12;21)\\
u_{10}/u_1&(121;22;22)\\
u_{11}/u_5&(210;12;21)\\
u_8/u_7&(000;1,-1;-1,1).
\end{array}
\]
None belongs to the dominant weight semigroup.  Therefore the indicated
prime divides neither factor on the right, proving the three gcd statements.
\end{proof}

\begin{proposition}
\label{prop:upper-contained-in-invariants}
One has $\Uneg\subseteq\UR$.
\end{proposition}

\begin{proof}
Let $f\in\Uneg$.  From the initial chart,
\[ f\in\UR[u_1^{-1},u_5^{-1},u_7^{-1}]. \]
In a reduced UFD expression for $f$, the only possible irreducible
denominator factors are therefore $u_1,u_5,u_7$.  On the other hand,
\[ f\in\UR[(u_1')^{-1},u_5^{-1},u_7^{-1}]. \]
By \cref{lem:four-chart-coprimality}, all three inverted elements in this
second localization are coprime to $u_1$; hence the $u_1$-adic exponent
of $f$ is nonnegative.  The $u_5'$- and $u_7'$-charts similarly show
that the $u_5$- and $u_7$-adic exponents are nonnegative.  Thus the
reduced denominator has no irreducible factor, and $f\in\UR$.
\end{proof}

\subsection{The middle algebra and the discriminant}
\label{subsec:middle-extension}

The theta functions are first considered with all frozen variables inverted.
Set
\[ P=\kk[y_1^{\pm1},y_2^{\pm1},y_3^{\pm1}], \qquad K_\zeta^{\mathrm{tor}}=\kk[u_2^{\pm1},u_3^{\pm1},u_4^{\pm1},u_{13}^{\pm1},\zeta^{\pm1}] \]
and
\[ K_u^{\mathrm{tor}} =\kk[u_2^{\pm1},u_3^{\pm1},u_4^{\pm1},u_{13}^{\pm1}]. \]
Let $\Sigma_{\mathrm{Mar}}^{\mathrm{prin}}$ be the principal-coefficient
seed with the Markov principal part.  The coefficient maps are
\[ P\xrightarrow{\ \varphi_\zeta\ }K_\zeta^{\mathrm{tor}} \xrightarrow{\ \operatorname{ev}_{-1}\ }K_u^{\mathrm{tor}}, \qquad \operatorname{ev}_{-1}(\zeta)=-1, \]
where, in normalized $y$-coordinates,
\begin{equation}
\label{eq:principal-coefficient-specialization}
 \varphi_\zeta(y_1)=\frac{u_3}{u_2u_4},\qquad
 \varphi_\zeta(y_2)=\frac{\zeta u_2u_{13}}{u_3},\qquad
 \varphi_\zeta(y_3)=\frac{u_4}{u_{13}}.
\end{equation}
Clearing the common frozen units in the three exchanges gives
\eqref{eq:zeta-family-exchanges}.

Let
\[ N_{\mathrm{mut}}=\ZZ e_1\oplus\ZZ e_2\oplus\ZZ e_3, \qquad L_\partial=\bigoplus_{i=4}^7\ZZ e_i, \qquad L_{\mathrm{fr}}=L_\partial\oplus\ZZ e_\zeta. \]
In the ordered frozen basis
$(e_4,e_5,e_6,e_7,e_\zeta)$, corresponding to
$(u_2,u_3,u_4,u_{13},\zeta)$, the monomial map
$\varphi_\zeta$ is induced by
\begin{equation}
\label{eq:coefficient-exponent-map}
 \begin{aligned}
 \kappa:\ZZ^3&\longrightarrow L_{\mathrm{fr}},\\
 \epsilon_1&\longmapsto(-1,1,-1,0,0),\\
 \epsilon_2&\longmapsto(1,-1,0,1,1),\\
 \epsilon_3&\longmapsto(0,0,1,-1,0).
 \end{aligned}
\end{equation}
Together with the identity on mutable directions, $\kappa$ gives the
commutative diagram
\begin{equation}
\label{eq:coefficient-lattice-diagram}
 \begin{array}{ccc}
 N_{\mathrm{mut}}\oplus\ZZ^3
 &\xrightarrow{\ \mathrm{id}\oplus\kappa\ }&
 N_{\mathrm{mut}}\oplus L_{\mathrm{fr}}\\[-1mm]
 \big\downarrow&&\big\downarrow\\[-1mm]
 N_{\mathrm{mut}}&\xrightarrow{\ \mathrm{id}\ }&N_{\mathrm{mut}},
 \end{array}
\end{equation}
where the vertical arrows forget coefficient directions.  Put
\[
 N_{u,\zeta}=N_{\mathrm{mut}}\oplus L_{\mathrm{fr}},
 \qquad
 N_u=N_{\mathrm{mut}}\oplus L_\partial.
\]
For $g\in N_u$, let $\widetilde\theta_{(g,0)}$ be the geometric-coefficient
theta function with zero $e_\zeta$-coordinate, and put
\[ \theta_g^-:=\operatorname{ev}_{-1} \bigl(\widetilde\theta_{(g,0)}\bigr). \]

\begin{lemma}[Change of coefficients for theta functions]
\label{lem:change-of-coefficients-theta}
After the frozen rescaling used to pass from normalized $y$-coordinates to
\eqref{eq:zeta-family-exchanges}, the map
$\mathrm{id}\oplus\kappa$ in
\eqref{eq:coefficient-lattice-diagram} is a linear morphism of seed data.
Normalize a principal-coefficient theta function and its image under this
change of coefficients to have the same initial mutable monomial.  Then
\begin{equation}
\label{eq:change-of-coefficients-theta-map}
 \varphi_\zeta\bigl(\vartheta^{\mathrm{prin}}_{(a,0)}\bigr)
 =\widetilde\theta_{(a,0)}
 \qquad(a\in N_{\mathrm{mut}}).
\end{equation}
Moreover, for $q\in L_{\mathrm{fr}}$, translation in a frozen direction gives
\begin{equation}
\label{eq:frozen-theta-translation}
 \widetilde\theta_{(a,q)}=X^q\widetilde\theta_{(a,0)}.
\end{equation}
Hence finiteness for principal coefficients implies finiteness for
$\Sigma_{u,\zeta}$ on the coefficient torus.
\end{lemma}

\begin{proof}
The exponent vectors in \eqref{eq:coefficient-exponent-map} are those
of the three monomials in
\eqref{eq:principal-coefficient-specialization}.  In the mutable-row
convention, the principal extended matrix is
$[B_{\mathrm{Mar}}\mid I_3]$, whereas
\[ B_{u,\zeta}=[B_{\mathrm{Mar}}\mid\kappa^{\mathsf T}]. \]
Thus, after clearing the common frozen units,
$\mathrm{id}\oplus\kappa$ satisfies the defining seed-datum identities
and is the linear morphism inducing $\varphi_\zeta$.  Functoriality of
theta functions under linear morphisms of seed data
\cite[Theorem~4.9]{CMMM} gives
\eqref{eq:change-of-coefficients-theta-map}; the stated normalization removes
the frozen monomial introduced by the rescaling.  A frozen character has no
walls, so multiplication by $X^q$ translates the theta index and gives
\eqref{eq:frozen-theta-translation}.  The finiteness statement follows.
\end{proof}

\begin{proposition}[Specialization on the coefficient torus]
\label{prop:specialized-theta-span}
Every $\widetilde\theta_{(g,r)}$ is a finite Laurent polynomial.  The
functions $\theta_g^-$, $g\in N_u$, are valuatively and linearly
independent, and their $\kk$-span is an algebra.  If
\[ \cL_t^{-,\mathrm{tor}} :=\cL_t^-[(u_2u_3u_4u_{13})^{-1}], \qquad t=\circ,1,5,7, \]
then
\[ \theta_g^-\in \bigcap_{t=\circ,1,5,7}\cL_t^{-,\mathrm{tor}} \qquad(g\in N_u). \]
\end{proposition}

\begin{proof}
For the once-punctured-torus Markov seed with principal coefficients, Zhou
proves the full Fock--Goncharov conjecture
\cite[Theorem~1.1]{Zhou}.  Thus every integral mutable label gives a finite
theta Laurent polynomial, the theta parametrization uses the full integral
tropical lattice, and the middle and canonical theta algebras coincide.  By
\cref{lem:change-of-coefficients-theta}, every theta function of
$\Sigma_{u,\zeta}$ on the coefficient torus is therefore finite.

Apply $\operatorname{ev}_{-1}$.  Every coefficient monomial in
$K_\zeta^{\mathrm{tor}}$ maps to a nonzero unit of
$K_u^{\mathrm{tor}}$, hence to a non-zero-divisor.  Cluster scattering
functions satisfy Assumption~3.9 of \cite{CMMM}, so
\cite[Theorem~3.10]{CMMM} shows that the finite theta functions retain their
mutable tropicalizations and remain valuatively and linearly independent.
Writing a label in $N_{u,\zeta}$ as $(g,r)$, with
$g\in N_u$ and $r\in\ZZ$ denoting the coefficient of $e_\zeta$,
translation in a frozen direction gives
\[ \operatorname{ev}_{-1}\bigl(\widetilde\theta_{(g,r)}\bigr) =(-1)^r\theta_g^-. \]
The image of the ordinary middle algebra, equal here to the canonical
algebra, is therefore the algebra spanned by the $\theta_g^-$ on the
coefficient torus.  Finally, every finite theta function is Laurent in every ordinary
cluster chart; specialization gives the asserted membership in the four Laurent rings over $K_u^{\mathrm{tor}}$.
\end{proof}

Let $D_4,D_5,D_6,D_7$ be the four frozen boundary divisors, and define
\[ \Mu :=\operatorname{span}_{\kk}\left\{ \theta_g^-\ \middle|\ \ord_{D_f}(\theta_g^-)\ge0\text{ for }4\le f\le7 \right\}. \]

\begin{proposition}
\label{prop:boundary-regular-theta-algebra}
The theta functions regular along all four divisors form a basis of $\Mu$, and
\[ \Mu\subseteq\Uneg. \]
\end{proposition}

\begin{proof}
By \cref{prop:specialized-theta-span}, the span of all specialized theta
functions is an algebra and the functions are valuatively independent.
Hence the order of a finite linear combination along any $D_f$ is the
minimum of the orders of its nonzero theta terms.  Such a combination is
regular along all four boundaries exactly when each theta term is regular.
This proves the basis assertion; closure under multiplication follows
because products of regular functions are regular.

On each cluster chart, $\ord_{D_f}$ is the minimum exponent of the frozen
variable $u_f$.  A theta function regular along all four $D_f$ has no
negative frozen exponent in any of the four Laurent charts.  The Laurent expansions from \cref{prop:specialized-theta-span} therefore
belong to the nonlocalized rings $\cL_t^-$.  Hence $\Mu\subseteq\Uneg$.
\end{proof}

We henceforth write $\theta_g=\theta_g^-$.
Let $e_1,\ldots,e_7$ be the standard basis in the variable order
\eqref{eq:markov-variable-order}.  The theta functions dual to the four frozen divisors have the following
exponent supports:
\begin{equation*}
 \begin{aligned}
 A_4={}&\{e_4,e_1+e_4,e_1+e_3+e_4,
          e_1+2e_3+e_4,e_1+e_2+2e_3+e_4\},\\
 A_5={}&\{e_5,e_2+e_5,e_1+e_2+e_5\},\\
 A_6={}&\{e_6,e_1+e_6,e_1+e_3+e_6\},\\
 A_7={}&\{e_7,e_3+e_7,e_2+e_3+e_7\}.
 \end{aligned}
\end{equation*}
The resulting cone is
\begin{equation}
\label{eq:explicit-cover-cone}
 \Xi=\left\{g\in\RR^7\ \middle|\
 \langle\alpha,g\rangle\ge0
 \text{ for every }\alpha\in A_4\cup A_5\cup A_6\cup A_7\right\}.
\end{equation}

\begin{lemma}\label{lem:boundary-support-computation}
After specialization at $\zeta=-1$, the theta functions dual to the four
frozen divisors have tropicalizations
\[ (\vartheta_f^\vee)^{\mathrm{trop}}(g) =\min_{\alpha\in A_f}\langle\alpha,g\rangle, \qquad 4\le f\le7. \]
The middle algebra therefore has theta basis
\[ \{\theta_g\mid g\in\Xi\cap\ZZ^7\}. \]
\end{lemma}

\begin{proof}
Work first in the ordinary family.  Mutation equivariance of broken lines
\cite[Proposition~3.6]{GHKK} transports the straight broken line representing a frozen theta function
from an adjacent chamber to the initial chamber.  The required mutation sequences and boundary-column entries are
\[
\begin{array}{c|c|c|l}
 f&\text{word}&\text{successive }b_{k,f}
 &X_f^{-1}\vartheta_f^\vee\\ \hline
 4&(1,3,2)&(-1,-2,1)
  &1+X_1+2X_1X_3+X_1X_3^2+X_1X_2X_3^2\\
 5&(2,1)&(-1,-1)&1+X_2+X_1X_2\\
 6&(1,3)&(-1,-1)&1+X_1+X_1X_3\\
 7&(3,2)&(-1,-1)&1+X_3+X_2X_3.
\end{array}
\]
Indeed, after projecting to the mutable exponent lattice, a bend of
multiplicity $r$ in direction $k$ replaces an exponent $\alpha$ by
the support of $X^\alpha(1+X_k)^r$.  The only bend of multiplicity two is
the middle step in the first row.  Thus the table gives the mutable supports
$A_f-e_f$ before coefficient specialization.

Exact monomial support need not survive coefficient specialization.  Only
the mutable tropicalization enters the boundary inequalities, and
\cite[Theorem~3.10]{CMMM} shows that it is unchanged under the
composite specialization used above.  Restoring the fixed boundary monomial
$X_f$ gives
\[ (\vartheta_f^\vee)^{\mathrm{trop}}(g) =\min_{\alpha\in A_f}\langle\alpha,g\rangle. \]
Theta reciprocity \cite[Theorem~5.19]{CMMM} identifies this tropicalization
with the boundary valuation:
\[ \ord_{D_f}(\theta_g) =\min_{\alpha\in A_f}\langle\alpha,g\rangle. \]
Thus $\theta_g$ is regular along every frozen boundary exactly when
$g\in\Xi$.  The basis assertion follows from
\cref{prop:boundary-regular-theta-algebra}.
\end{proof}

\begin{lemma}
\label{lem:Xi-Hilbert-basis}
The Hilbert basis of $\Xi\cap\ZZ^7$ consists of the following sixteen
vectors:
\begin{equation}
\label{eq:Xi-Hilbert-basis}
 \begin{aligned}
 \mathcal H_1={}&\{e_1,e_2,e_3,2e_1-e_3+e_7\},\\
 \mathcal H_0={}&\{e_4,e_5,e_6,e_7,
 -e_1+e_2+e_4+e_6,
 -e_2+e_3+e_5, e_1-e_2+e_5+e_7,
 e_1-e_3+e_4+e_7\},\\
 \mathcal H_{-1}={}&\{-e_1+e_4+e_5+e_6,
 -e_2+e_4+e_5+e_7, -e_3+2e_4+e_6+e_7,
 -2e_2+e_3+2e_5+e_7\}.
 \end{aligned}
\end{equation}
The subscript is the mutable degree $g_1+g_2+g_3$.
\end{lemma}

\begin{proof}
Write $g=(m_1,m_2,m_3,d_4,d_5,d_6,d_7)$.  The inequalities defining
$\Xi$ are equivalent to
\[
 \begin{aligned}
 d_4&\ge\max\{0,-m_1,-m_1-m_3,-m_1-2m_3,
                    -m_1-m_2-2m_3\},\\
 d_5&\ge\max\{0,-m_2,-m_1-m_2\},\\
 d_6&\ge\max\{0,-m_1,-m_1-m_3\},\\
 d_7&\ge\max\{0,-m_3,-m_2-m_3\}.
 \end{aligned}
\]
Let $\psi\colon\RR^3\to\RR^4$ be the vector of the four maximum
functions.  Then $\Xi$ is the integral epigraph of $\psi$: every point
is a minimal lift $(m,\psi(m))$, plus a nonnegative combination of the
vertical rays $e_4,\ldots,e_7$.

A comparison of the affine forms shows that the common fan of linearity of
$\psi$ is obtained from the orthant fan by the successive star
subdivisions along
\[
 \begin{gathered}
 \langle-e_1,e_2\rangle,\qquad
 \langle e_1,-e_3\rangle,\qquad
 \langle e_1,e_1-e_3\rangle,\\
 \langle-e_2,e_3\rangle,\qquad
 \langle e_1,-e_2\rangle,\qquad
 \langle-e_2,-e_2+e_3\rangle.
 \end{gathered}
\]
The order matters: each cone is present when it is subdivided.
Every new ray is the sum of the primitive generators of a unimodular
two-dimensional cone, so unimodularity is preserved.  The resulting complete fan has twenty maximal cones; $\psi$ is linear on
each cone.  The \texttt{fan}
task in \path{Kron322.py} lists the cones and checks their
determinants, active linear forms, and twelve nonvertical lifts.

On a unimodular cone every lattice point is a nonnegative integral
combination of its primitive rays, and $\psi$ is linear.  Taking minimal
lifts therefore expresses every lattice point of $\Xi$ as a nonnegative
integral combination of the sixteen vectors.  A minimal lift of a
primitive fan ray is indecomposable: equality in the subadditivity
inequalities would force both projected summands into one cone of linearity
and hence onto the same extremal ray; primitivity then forces one summand to
have zero projection, and minimality forces it to vanish.  The four vertical
unit rays are indecomposable.  Hence the sixteen vectors form the
Hilbert basis.
\end{proof}

To pass from theta indices to triple weights, order the rows by
$(u_1,u_5,u_7;u_2,u_3,u_4,u_{13})$ and the columns by
$(\lambda_1,\lambda_2,\lambda_3;\mu_1,\mu_2;\nu_1,\nu_2)$.  The
resulting weight matrix is
\begin{equation}
\label{eq:weight-matrix-W}
 W=
 \left[\begin{array}{ccc|cc|cc}
 1&0&0&1&0&1&0\\
 1&1&1&2&1&2&1\\
 2&1&1&2&2&3&1\\
 1&1&0&1&1&2&0\\
 1&1&0&2&0&1&1\\
 2&0&0&1&1&1&1\\
 2&2&2&3&3&3&3
 \end{array}\right].
\end{equation}
Thus $\wt(\theta_g)=gW$.  The cone $\Xi$ is pointed: if both
$g$ and $-g$ lie in $\Xi$, the inequalities coming from
\[ e_4,e_5,e_6,e_7,e_1+e_4,e_2+e_5,e_3+e_7 \]
force all seven coordinates of $g$ to vanish.  Since
$\Mu\subseteq\Uneg\subseteq\UR$ by
\cref{prop:boundary-regular-theta-algebra,prop:upper-contained-in-invariants},
every triple-weight component of $\Mu$ is contained in the corresponding
finite-dimensional Kronecker weight space of $\UR$.  For the lowest-term
seed valuation fixed above, $\initer(\theta_g)=\mathbf x^g$, and its leaves
are one-dimensional.  Hence
\cref{lem:hilbert-theta-generation,lem:Xi-Hilbert-basis} shows that the sixteen Hilbert-basis theta functions form a Khovanskii
basis of $\Mu$ and, in particular, generate it.

\begin{proposition}
\label{prop:twelve-generators-in-middle}
Write $f\doteq g$ when two nonzero homogeneous functions differ by a
scalar in $\kk^\times$.  For $g\in\mathcal H_1\cup\mathcal H_0$, the
table records the correspondence $\theta_g\doteq u_i$.
\begingroup
\setlength{\arraycolsep}{3pt}
\renewcommand{\arraystretch}{1.12}
\begin{equation}
\label{eq:theta-generator-table}
\begin{array}{cc|cc@{\qquad}cc}
\multicolumn{2}{c|}{\mathcal H_1}
 &\multicolumn{4}{c}{\mathcal H_0}\\
 g&\theta_g&g&\theta_g&g&\theta_g\\ \hline
 e_1&u_1&e_4&u_2&-e_1+e_2+e_4+e_6&u_{11}\\
 e_2&u_5&e_5&u_3&-e_2+e_3+e_5&u_6\\
 e_3&u_7&e_6&u_4&e_1-e_2+e_5+e_7&u_{12}\\
 2e_1-e_3+e_7&u_8&e_7&u_{13}&e_1-e_3+e_4+e_7&u_{10}
\end{array}
\end{equation}
\endgroup
Thus every generator of $\UR$ except $u_9$ belongs to $\Mu$.
\end{proposition}

\begin{proof}
For the degree-one labels and the four frozen basis vectors, the assertion
follows from their $g$-vectors.  The remaining four theta functions lie in
$\Mu\subseteq\Uneg\subseteq\UR$.  Their $W$-weights, in the order shown in
the last two columns of the table, are
\[ (321;33;42),\quad(210;21;21),\quad (321;42;33),\quad(221;32;32). \]
The corresponding invariant weight spaces are one-dimensional by
\cref{lem:relevant-one-dimensional-weights}; hence each nonzero theta function
is a scalar multiple of the indicated normalized invariant.  The scalars play
no role in algebra generation.
\end{proof}

Put
\[ \udisc:=u_9=u_{220,22}^{22}, \qquad \chi_\Delta=\wt(\udisc)=(2,2,0;2,2;2,2). \]

\begin{lemma}\label{lem:discriminant-chart-identities}
The discriminant satisfies
\begin{equation}
\label{eq:discriminant-chart-identities}
 \begin{aligned}
 u_1^2\udisc&=u_6^2+4u_2u_3u_4,\\
 u_5^2\udisc&=u_{10}^2-4u_2u_3u_{13}.
 \end{aligned}
\end{equation}
\end{lemma}

\begin{proof}
The first identity is the second relation in
\eqref{eq:basic-relations}.  The first and third relations there give
\[ u_1u_{10}=u_2u_8+u_3u_7, \qquad u_5u_6=u_3u_7-u_2u_8. \]
Hence, using also the last basic relation,
\[
 \begin{aligned}
 u_1^2\bigl(u_{10}^2-\udisc u_5^2\bigr)
 &=(u_2u_8+u_3u_7)^2-u_5^2(u_6^2+4u_2u_3u_4)\\
 &=4u_2u_3(u_7u_8-u_4u_5^2)
  =4u_1^2u_2u_3u_{13}.
 \end{aligned}
\]
Canceling $u_1^2$ in the domain $\UR$ proves the second identity.
\end{proof}

\begin{theorem}
\label{thm:middle-plus-discriminant}
One has
\begin{equation}
\label{eq:upper-middle-extension}
 \UR=\Uneg=\Mu[\udisc].
\end{equation}
Moreover, $\udisc\notin\Mu$.
\end{theorem}

\begin{proof}
By \cref{prop:twelve-generators-in-middle}, the twelve generators
$u_i$ with $i\ne9$ belong to $\Mu\subseteq\Uneg$.  It remains to show that $\udisc$ belongs to all four Laurent charts.  In
$\cL_\circ^-$, $\cL_5^-$, and $\cL_7^-$, the element $u_1$ is
invertible, so the first identity in
\eqref{eq:discriminant-chart-identities} expresses $\udisc$ as a Laurent
function there.  In $\cL_1^-$, the element $u_5$ is invertible, and the
second identity does the same.  Thus $\udisc\in\Uneg$.

The thirteen-generator theorem and
\cref{prop:boundary-regular-theta-algebra,prop:upper-contained-in-invariants}
give the inclusions
\[ \UR=\kk[u_1,\ldots,u_{13}] \subseteq\Mu[\udisc] \subseteq\Uneg \subseteq\UR, \]
which proves \eqref{eq:upper-middle-extension}.

For the final assertion, solve $gW=\chi_\Delta$.  Every solution has the form
\[ g=(2t-2,\,2s-2,\,2-2s-2t,\,s+t,\,2-s-t,\,s,\,t). \]
The inequalities $g_6,g_7\ge0$,
$g_2+g_3+g_7\ge0$,
$g_1+g_4\ge0$, and $g_3+g_7\ge0$, all belonging to
\eqref{eq:explicit-cover-cone}, force successively
$s,t\ge0$, $t=0$, $s\ge2$, and $s\le1$, a contradiction.
Hence there is no theta function in $\Mu$ of weight $\chi_\Delta$.  Since
the theta functions form a homogeneous basis, the corresponding weight
space is zero, whereas $\udisc\ne0$.
\end{proof}

\section{A polyhedral formula}
\label{sec:polyhedral}

\subsection{A pairing of theta generators}

Define the \emph{mutable degree} to be the $\ZZ$-grading determined by
\[ \deg u_1=\deg u_5=\deg u_7=1, \qquad \deg u_2=\deg u_3=\deg u_4=\deg u_{13}=0. \]
The three exchange relations are homogeneous for this assignment, so mutable
degree defines a $\ZZ$-grading of the cluster family, its fiber at
$\zeta=-1$, and $\Mu[\udisc]$.  In initial Laurent coordinates,
\[ \deg\theta_g=g_1+g_2+g_3, \qquad \deg\udisc=-2. \]
Order the positive- and negative-degree parts of the
Hilbert basis in \eqref{eq:Xi-Hilbert-basis} as
\[
\begin{aligned}
 h_1^+&=e_1,&
 h_1^-&=-e_1+e_4+e_5+e_6,&
 h_2^+&=e_2,&
 h_2^-&=-e_2+e_4+e_5+e_7,\\
 h_3^+&=e_3,&
 h_3^-&=-e_3+2e_4+e_6+e_7,&
 h_4^+&=2e_1-e_3+e_7,&
 h_4^-&=-2e_2+e_3+2e_5+e_7.
\end{aligned}
\]
By \eqref{eq:theta-generator-table},
\[ \theta_{h_1^+}\doteq u_1,\qquad \theta_{h_2^+}\doteq u_5,\qquad \theta_{h_3^+}\doteq u_7,\qquad \theta_{h_4^+}\doteq u_8. \]

\begin{proposition}\label{prop:paired-theta-generators}
For $1\le i\le4$, there is a scalar $c_i\in\kk^\times$ such that
\begin{equation}
\label{eq:paired-theta-generators}
 \theta_{h_i^-}=c_i\udisc\,\theta_{h_i^+}.
\end{equation}
Up to nonzero scalars, the four generators of degree $-1$ are
$\udisc u_1,\udisc u_5,\udisc u_7,\udisc u_8$.
\end{proposition}

\begin{proof}
Multiplication by the weight matrix gives
\[ h_i^-W=h_i^+W+\chi_\Delta. \]
More explicitly, the four common weights and the corresponding monomials are
\[
\begin{array}{c|c|c}
 i&h_i^-W&\text{monomial of this weight in }\kk[u_1,\ldots,u_{13}]\\ \hline
 1&(320;32;32)&\udisc u_1\\
 2&(331;43;43)&\udisc u_5\\
 3&(431;44;53)&\udisc u_7\\
 4&(431;53;44)&\udisc u_8.
\end{array}
\]
By \cref{lem:relevant-one-dimensional-weights}, each monomial
spans the corresponding one-dimensional weight space of $\UR$.  Both
$\theta_{h_i^-}$ and
$\udisc\theta_{h_i^+}$ are nonzero elements of that space, which proves
\eqref{eq:paired-theta-generators}.
\end{proof}

\begin{corollary}
\label{cor:multiplication-by-discriminant}
Let $f\in\Mu$ be homogeneous of mutable degree $d$.  If
$0\le r\le d$ or $d \le r \le 0$, then
\[ \udisc^rf\in\Mu. \]
\end{corollary}
\begin{proof}
By \cref{lem:hilbert-theta-generation,lem:Xi-Hilbert-basis}, the algebra
$\Mu$ is generated by eight degree-zero elements, the four
$\theta_{h_i^+}$, and the four $\theta_{h_i^-}$.  Express $f$ as a
homogeneous polynomial in these generators.  For a monomial, let $N_+$ and $N_-$ denote the numbers of degree $1$ and degree $-1$ factors, counted with multiplicity.  Then
$N_+-N_-=d$.  If $d\ge0$, then $N_+\ge d\ge r$, so $r$ degree $1$ factors
may be replaced by their degree $-1$ partners.  If $d<0$, then
$N_-\ge-d\ge-r$, so $-r$ degree $-1$ factors may be replaced by their
degree $1$ partners.  By \eqref{eq:paired-theta-generators}, these
replacements multiply the monomial by $\udisc^r$ up to a nonzero scalar.  Applying the replacement to every
monomial in a homogeneous expression for $f$ proves the claim.
\end{proof}

\subsection{Reduction to the \texorpdfstring{$\ell_0$}{ell0}-slice}

Let
\[ \chi=(\lambda_1,\lambda_2,\lambda_3; \mu_1,\mu_2;\nu_1,\nu_2) \]
be an admissible triple weight.  Put
\begin{equation}
\label{eq:A-and-level}
 A=\lambda_1+2\lambda_2-\mu_1-\nu_1,
 \qquad
 \ell_0=\max\left\{0,\left\lfloor\frac{A+1}{2}\right\rfloor\right\}.
\end{equation}
A weight space with a nondominant triple weight is understood to be zero, and
the corresponding cone slice below is empty.  If $\chi-\ell_0\chi_\Delta$ is not dominant, both sides of
\cref{eq:polyhedral-Kronecker-count} are zero.

\begin{theorem}
\label{thm:polyhedral-Kronecker}
The Kronecker coefficient of $\chi$ is the number of lattice points in a
single slice of the cone:
\begin{equation}
\label{eq:polyhedral-Kronecker-count}
 g_{\lambda,\mu,\nu}
 =\#\left\{g\in\Xi\cap\ZZ^7\ \middle|\
        gW=\chi-\ell_0\chi_\Delta\right\}.
\end{equation}
\end{theorem}

\begin{proof}
Calculation from \eqref{eq:weight-matrix-W} shows that every theta
function of weight $\chi-\ell\chi_\Delta$ has mutable degree
\begin{equation}
\label{eq:level-degree}
 d_\ell=2\ell-A.
\end{equation}
By \cref{lem:kronecker-weight-space,thm:middle-plus-discriminant}, the desired
Kronecker coefficient is the dimension of
\[ \Mu[\udisc]_\chi =\sum_{\ell\ge0}\udisc^\ell \Muwt{\chi-\ell\chi_\Delta}. \]
This sum is finite in a fixed triple weight, because
$\chi-\ell\chi_\Delta$ is nondominant for $\ell\gg0$.  Moreover, every
vector in $\Muwt{\chi-\ell\chi_\Delta}$ has the same mutable degree
$d_\ell$, since \eqref{eq:level-degree} expresses mutable degree as a
linear function of the triple weight.  We show that every summand is contained in the summand at level $\ell_0$.

Suppose first that $\ell\ge\ell_0$. Put $r=\ell-\ell_0$. Since $2\ell_0\ge A$,
\[ d_\ell-r=\ell+\ell_0-A\ge0. \]
Thus $0\le r\le d_\ell$, and
\cref{cor:multiplication-by-discriminant} gives
\[ \udisc^r\Muwt{\chi-\ell\chi_\Delta} \subseteq\Muwt{\chi-\ell_0\chi_\Delta}. \]
Multiplication by $\udisc^{\ell_0}$ yields
\[ \udisc^\ell\Muwt{\chi-\ell\chi_\Delta} \subseteq \udisc^{\ell_0}\Muwt{\chi-\ell_0\chi_\Delta}. \]

Suppose now that $\ell<\ell_0$.  Then $\ell_0\ge\ell+1$ and
$A\ge2\ell_0-1$, so
\[ d_\ell=2\ell-A\le2(\ell_0-1)-A<0. \]
By \cref{cor:multiplication-by-discriminant},
\[ \udisc^{d_\ell}\Muwt{\chi-\ell\chi_\Delta} \subseteq \Muwt{\chi-(\ell-d_\ell)\chi_\Delta}. \]
Multiplying by $\udisc^{\ell-d_\ell}$ gives
\[ \udisc^\ell\Muwt{\chi-\ell\chi_\Delta} \subseteq \udisc^{\ell-d_\ell} \Muwt{\chi-(\ell-d_\ell)\chi_\Delta}. \]
Set $\ell^\vee=\ell-d_\ell=A-\ell$.  Since
$\ell^\vee\ge A-(\ell_0-1)\ge\ell_0$, the preceding case applied to
$\ell^\vee$ places this summand in the summand at level $\ell_0$.

It follows that
\[ \Mu[\udisc]_\chi =\udisc^{\ell_0}\Muwt{\chi-\ell_0\chi_\Delta}. \]
Multiplication by $\udisc^{\ell_0}$ is injective because the ambient
invariant ring is a domain.  The theta basis of $\Mu$ is homogeneous and
linearly independent, so the dimension of the last weight space is the
lattice-point count in \eqref{eq:polyhedral-Kronecker-count}.
\end{proof}

\subsection{The finite-sum formula}

Assume that the shifted triple is dominant; otherwise the coefficient is
zero by \cref{thm:polyhedral-Kronecker}.  Set
\begin{equation}
\label{eq:shifted-weight}
 \bar\chi=\chi-\ell_0\chi_\Delta
 =(L_1,L_2,L_3;M_1,M_2;N_1,N_2).
\end{equation}
Thus
\[ L_1+L_2+L_3=M_1+M_2=N_1+N_2. \]
Interchanging the last two partitions if necessary, assume
\[ M_1-M_2\le N_1-N_2, \]
which is equivalent to $M_2\ge N_2$.  Put
\begin{equation}
\label{eq:closed-form-parameters}
 \begin{aligned}
 a&=L_1-L_2,& b&=L_2-L_3,& c&=L_3,\\
 \rho&=M_1-M_2,& h&=M_2-N_2,&d&=M_1+N_1-L_1-2L_2.
 \end{aligned}
\end{equation}
The integer $d$ is the mutable degree of the shifted weight and is
nonnegative by the definition of $\ell_0$.

\begin{theorem}
\label{thm:closed-formula}
For $x\in\ZZ$, set $m_x=b-h-2x$, and
\begin{equation}
\label{eq:closed-sum-bounds}
 \begin{aligned}
 \mathsf L(x)&=\max\left\{
 0,\left\lceil\frac{a-d-x}{2}\right\rceil,N_2-2c-x
 \right\},\\
 \mathsf U(x)&=\min\{0,m_x\}
 +\min\left\{
 a,\left\lfloor\frac{a+h+x}{2}\right\rfloor,M_2-L_2+x
 \right\}.
 \end{aligned}
\end{equation}
Then
\begin{equation}
\label{eq:closed-Kronecker-sum}
 g_{\lambda,\mu,\nu}
 =\sum_{x=0}^{\min\{b,\rho\}}
   \pos{\mathsf U(x)-\mathsf L(x)+1}.
\end{equation}
Consequently, on the lattice of admissible triple weights, these Kronecker
coefficients are piecewise quasi-polynomial with quasiperiod dividing two.
\end{theorem}

\begin{proof}
To count the points in \eqref{eq:polyhedral-Kronecker-count}, put
$\delta=-g$, so the cone inequalities become
$\langle\alpha,\delta\rangle\le0$ for all
$\alpha\in A_4\cup A_5\cup A_6\cup A_7$, and solve
$-\delta W=\bar\chi$.  The integral solution set is an affine lattice of rank two.  Projection
to
\[ x=-\delta_5, \qquad y=-\delta_6 \]
is an affine lattice isomorphism onto $\ZZ^2$: solving the remaining five
equations over $\ZZ$ gives the inverse
\begin{equation*}
 \begin{aligned}
 \delta_1&=-a-b+h+2x+2y,&
 \delta_2&=a-d-2y,&
 \delta_3&=b-h-2x,\\
 \delta_4&=x-b,\qquad
 \delta_5=-x,&
 \delta_6&=-y,&\delta_7&=x+y+c-N_2.
 \end{aligned}
\end{equation*}
Substitution in the fourteen cone inequalities first gives
\[ 0\le x\le\min\{b,\rho\}. \]
For a fixed $x$, the nonredundant lower bounds on $y$ are
\[ y\ge0, \qquad y\ge\left\lceil\frac{a-d-x}{2}\right\rceil, \qquad y\ge N_2-2c-x, \]
whose maximum is $\mathsf L(x)$.

For the upper bounds, put
\[ C_x=\left\lfloor\frac{a+h+x}{2}\right\rfloor. \]
The six nonredundant inequalities are
\[ y\le a,\quad y\le C_x,\quad y\le M_2-L_2+x, \]
\[ y\le a+m_x,\quad y\le C_x+m_x, \quad y\le M_2-L_2+x+m_x. \]
Their minimum is $\mathsf U(x)$.  The remaining apparent bound
\[ y\le\left\lfloor\frac{a+b-x}{2}\right\rfloor \]
is redundant: before taking floors, its right-hand side is the arithmetic
mean of
\[ \frac{a+h+x}{2} \quad\text{and}\quad \frac{a+2b-h-3x}{2}. \]
Thus, for fixed $x$, the number of integral values of $y$ is
$\pos{\mathsf U(x)-\mathsf L(x)+1}$.  Summing over the allowed values of
$x$ proves \eqref{eq:closed-Kronecker-sum}.

Only division by two occurs in the endpoints.  Subdivide the admissible
parameter space so that the active affine branches in
\eqref{eq:closed-sum-bounds}, the order of their finitely many crossover
points of $x$, and the affine zeros of
$\mathsf U(x)-\mathsf L(x)+1$ are fixed.  On each resulting chamber both
the summation interval and the active branch of the positive-part function are fixed.  The
sum is therefore a finite sum of affine functions over intervals with
parity-dependent endpoints, hence a quasi-polynomial with quasiperiod
dividing two.
\end{proof}

\begin{remark}
The quasiperiod bound is sharp, although it need not be minimal on every
chamber.  For $r\ge0$ and $\lambda=\mu=\nu=(r,r)$, one has $A=r$ and
$\ell_0=\lceil r/2\rceil$.  If $r$ is even, the shifted weight is zero
and the slice consists of the origin; if $r$ is odd, the shifted weight is
nondominant.  Therefore
\[ g_{(r,r),(r,r),(r,r)}=\frac{1+(-1)^r}{2}. \]
Exact period two occurs on this ray, while on other chambers the parity
dependence may cancel and the quasi-polynomial may have period one.
\end{remark}

The \texttt{finite-sum} task in \path{Kron322.py} verifies the affine
inverse and the fourteen substituted inequalities, and compares the formula
with the symmetric-group character formula up to a user-specified size (by
default, twelve).

\begin{corollary}[All $n\times2\times2$ formats]
\label{cor:closed-formula-22n}
Let $n\ge1$, and let $\lambda,\mu,\nu$ be partitions of a common integer
$N$ satisfying
$\ell(\lambda)\le n$ and $\ell(\mu),\ell(\nu)\le2$.
If $\ell(\lambda)>4$, then $g_{\lambda,\mu,\nu}=0$.  Otherwise pad
$\lambda$ to four parts and put $r=\lambda_4$.  If
$\mu_2<2r$ or $\nu_2<2r$, then again
$g_{\lambda,\mu,\nu}=0$.  In all remaining cases set
\[ \bar\lambda=(\lambda_1-r,\lambda_2-r,\lambda_3-r), \qquad \bar\mu=(\mu_1-2r,\mu_2-2r), \qquad \bar\nu=(\nu_1-2r,\nu_2-2r). \]
Apply \cref{eq:A-and-level,eq:shifted-weight,eq:closed-form-parameters,eq:closed-sum-bounds}
to $(\bar\lambda,\bar\mu,\bar\nu)$.  Then
\[ g_{\lambda,\mu,\nu} =\sum_{x=0}^{\min\{b,\rho\}} \pos{\mathsf U(x)-\mathsf L(x)+1}. \]
The formula therefore applies to every Kronecker coefficient of format
$n\times2\times2$.
\end{corollary}

\begin{proof}
If $\ell(\lambda)>4$, then
$\Schur_\lambda(\kk^2\otimes\kk^2)=0$, so the coefficient vanishes.
Assume $\ell(\lambda)\le4$.  Proposition~\ref{prop:det-extension}, equivalently
\cref{cor:rectangular-translation}, with $b=c=2$ and
$r=\lambda_4$ gives zero when $\mu_2<2r$ or $\nu_2<2r$, and
otherwise gives
\[ g_{\lambda,\mu,\nu}=g_{\bar\lambda,\bar\mu,\bar\nu}. \]
The first partition on the right has length at most three, so
\cref{thm:closed-formula} applies.  For $n\ge4$,
\cref{cor:stable-first-factor} shows that no further dependence on the first
ambient dimension occurs.  For $n\le3$, enlarging the first vector space to
$\kk^3$ and padding by zeros leaves the same highest-weight multiplicity.
\end{proof}

\appendix
\raggedbottom

\section{The cluster algebra on the \texorpdfstring{double-$U$}{double-U}-invariant ring}
\label{app:partial-ring}

Only the part of the flagged-Kronecker construction used above is recalled
here.  For positive integers $l_1,l_2,m$, let $K_{l_1,l_2}^m$ be the
flagged $m$-arrow Kronecker quiver
\[ -1\longrightarrow-2\longrightarrow\cdots\longrightarrow-l_1 \overset{m}{\Longrightarrow} l_2\longrightarrow l_2-1\longrightarrow\cdots\longrightarrow1 \]
with dimension vector $\beta(i)=|i|$.  Its semi-invariant ring is
\[ \SI_\beta(K_{l_1,l_2}^m) =\kk[\Rep_\beta(K_{l_1,l_2}^m)]^{\prod_i\SL_{\beta(i)}}. \]
Contracting the two equioriented flag arms gives a multigraded,
$\GL_m$-equivariant isomorphism
\begin{equation}
\label{eq:flagged-Kronecker-transfer}
 \SI_\beta(K_{l_1,l_2}^m)
 \cong
 \kk[\kk^{l_1}\otimes\kk^{l_2}\otimes\kk^m]^{U_{l_1}\times U_{l_2}}.
\end{equation}
One way to see \eqref{eq:flagged-Kronecker-transfer} is to tensor the
coordinate ring of the central representation space with the coordinate
rings of the two basic affine spaces $\GL_{l_i}/U_{l_i}$, and then take
$\GL_{l_1}\times\GL_{l_2}$-invariants.  Equivalently, successive Cauchy
decompositions along the arms retain exactly the highest-weight line at each
central vertex.

For $(l_1,l_2,m)=(3,2,2)$, the right-hand side is
$\SR$.  In the variable order
\eqref{eq:double-initial-cluster}, its extended exchange matrix is
\begin{equation*}
 B_s=
 \begin{bmatrix}
  0& 1&-1&-1& 1& 0& 1&-1\\
 -1& 0& 2& 1& 0&-1& 0& 0\\
  1&-2& 0& 0&-1& 1& 0& 0
 \end{bmatrix}.
\end{equation*}
The rows are indexed by $s_1,s_2,s_3$; the columns are ordered by
all eight variables in \eqref{eq:double-initial-cluster}, with the last five
columns frozen.  The results
of \cite{Fk1,Fk2} give
\[ \SR=\Asup. \]

Let $\epsilon_1,\ldots,\epsilon_8$ be the standard basis in the order
\eqref{eq:double-initial-cluster}.  The noninitial primitive rays of the
$g$-vector cone used here are
\begin{equation*}
 \begin{aligned}
 \gamma_1'&=\epsilon_1-\epsilon_3+\epsilon_6,&
 \gamma_2'&=-\epsilon_2+2\epsilon_3+\epsilon_4,\\
 \gamma_2^\circ&=-\epsilon_2+\epsilon_3+\epsilon_4,&
 \gamma_3'&=-\epsilon_1+\epsilon_2+\epsilon_5+\epsilon_7,\\
 \gamma_{31}&=-\epsilon_3+\epsilon_6+\epsilon_7,&
 \gamma_{32}&=-\epsilon_1+\epsilon_3+\epsilon_5+\epsilon_7,\\
 \gamma_{312}&=-\epsilon_1+\epsilon_5+2\epsilon_7.
 \end{aligned}
\end{equation*}
Together with the five frozen rays, these generate the integral
$g$-vector cone.  The labels used in the main text are fixed by the
following table.
\begin{table}[ht]
\centering
\small
\caption{The fourteen generators of $\SR$ used in the two-slice proof.}
\label{tab:partial-generators}
\begin{tabular}{ccll}
\toprule
$i$&$g(s_i)$&$\wt(s_i)$&$\partial(s_i)$\\ \midrule
1&$\epsilon_1$&$(110;20;11)$&0\\
2&$\epsilon_2$&$(100;10;01)$&$s_3$\\
3&$\epsilon_3$&$(100;10;10)$&0\\
4&$\epsilon_4$&$(110;11;02)$&$s_9$\\
5&$\epsilon_5$&$(110;11;20)$&0\\
6&$\epsilon_6$&$(200;11;11)$&0\\
7&$\epsilon_7$&$(111;21;12)$&$-s_8$\\
8&$\epsilon_8$&$(111;21;21)$&0\\ \midrule
9&$\gamma_2^\circ$&$(110;11;11)$&$2s_5$\\
10&$\gamma_3'$&$(211;22;22)$&$s_{11}$\\
11&$\gamma_{32}$&$(211;22;31)$&0\\
12&$\gamma_{31}$&$(211;22;13)$&$2s_{10}$\\
13&$\gamma_1'$&$(210;21;12)$&$2s_5s_2$\\
14&$\gamma_{312}$&$(222;33;33)$&0\\
\bottomrule
\end{tabular}
\end{table}

\begin{proposition}[{\cite[Proposition~9.5]{Fk1}}]
\label{prop:partial-minimal-generators}
The algebra $\SR$ is minimally generated by the eight initial
extended cluster variables and the five generic cluster characters with
$g$-vectors
\[ \gamma_2^\circ,\quad\gamma_3',\quad \gamma_{31},\quad\gamma_{32},\quad\gamma_{312}. \]
\end{proposition}

The fourteen generators used in \cref{subsec:invariant-upper} are obtained
from this minimal set by adjoining the one-step cluster variable with
$g$-vector $\gamma_1'$.  This additional variable shortens the root strings
in \eqref{eq:partial-on-s-generators}.

\section{A regular-sequence computation}
\label{app:regular-sequence-certificate}

\subsection{The single colon calculation}
\label{app:regular-sequence}

Let $S=\kk[U_1,\ldots,U_{13}]$, a Noetherian polynomial ring, and let
$J\subset S$ be generated by
\begin{equation}
\label{eq:J-six-relations}
 \begin{aligned}
 &U_3U_7-U_2U_8-U_5U_6,\\
 &U_1^2U_9-U_6^2-4U_2U_3U_4,\\
 &U_1U_{10}-2U_2U_8-U_5U_6,\\
 &U_1U_{11}+2U_2U_4U_5+U_6U_7,\\
 &U_1U_{12}-2U_3U_4U_5+U_6U_8,\\
 &U_1^2U_{13}-U_7U_8+U_4U_5^2.
 \end{aligned}
\end{equation}
Let $\varphi:S\to\Ugen$ send $U_i$ to $u_i$, and put
$I=\ker\varphi$.  The first slice in
\eqref{eq:slice-retraction-table} shows that, after inverting $U_1$, the
relations \eqref{eq:J-six-relations} give the complete presentation: the
last five relations eliminate $U_9,U_{10},U_{11},U_{12},U_{13}$, and the
first relation is the sole relation among the remaining eight generators.
Consequently, $I S_{U_1}=J S_{U_1}$.  Since
$S/I=\Ugen$ is a domain and $u_1\ne0$, the ideal $I$ is
$U_1$-saturated.  Contracting from $S_{U_1}$ therefore gives
\[ I=J S_{U_1}\cap S=J:U_1^\infty. \]
Because $S$ is Noetherian, the ascending union defining
$J:U_1^\infty$ stabilizes after finitely many colon operations.

Use graded reverse lexicographic order with variable order
\[ U_7>U_6>U_5>U_{13}>U_8>U_4>U_{10}>U_{11}> U_3>U_9>U_{12}>U_2>U_1. \]
The ancillary file
\texttt{Kron322\_regular\_sequence.m2} constructs
$I=J:U_1^\infty$, puts $K=I+(U_1)$, and verifies the identity
\[ \operatorname{trim}(K:U_2)=K. \]
Equivalently,
\[ (I+(U_1)):U_2=I+(U_1). \]
The Macaulay2 file records the software version and verifies this equality by
an explicit assertion.  The \texttt{regular-sequence} task in
\path{Kron322.py} independently performs the two saturations by elimination
over $\mathbb Q$ and verifies the stronger equality
\[ K:U_2^\infty=K. \]
The accompanying output records both elimination steps and the equality of
the reduced Gr\"obner bases.  Since the ideals and variables involved are defined over $\mathbb Q$, and every
characteristic-zero field extension of $\mathbb Q$ is flat, the colon
identity persists after scalar extension to $\kk$.

It follows that multiplication by $U_2$ is injective on
\[ S/(I,U_1)\cong\Ugen/u_1\Ugen. \]
Thus $u_1,u_2$ is a regular sequence in $\Ugen$, proving
\cref{lem:Ugen-regular-sequence}.

\section{The folded Donaldson--Thomas interpretation}
\label{app:folded-DT}

This appendix identifies the cluster transformation underlying the pairing in
\cref{prop:paired-theta-generators}.  It is not used in the proof of the
polyhedral or closed formulas.

Two consequences of \eqref{eq:basic-relations} will be used:
\begin{equation}
\label{eq:DT-two-identities}
 \begin{aligned}
 u_1^2u_2u_{13}+u_5^2u_2u_4+u_7^2u_3
   &=u_1u_7u_{10},\\
 u_{10}^2-\udisc u_5^2&=4u_2u_3u_{13}.
 \end{aligned}
\end{equation}
Indeed, the first and third relations in \eqref{eq:basic-relations} give
$u_1u_{10}=u_2u_8+u_3u_7$.  Multiplying by $u_7$ and using the last
relation proves the first identity.  For the second, use
$u_5u_6=u_3u_7-u_2u_8$, the second relation, and then the last relation:
\[
 \begin{aligned}
 u_1^2\bigl(u_{10}^2-\udisc u_5^2\bigr)
 &=(u_2u_8+u_3u_7)^2-u_5^2(u_6^2+4u_2u_3u_4)\\
 &=4u_2u_3(u_7u_8-u_4u_5^2)
  =4u_1^2u_2u_3u_{13}.
 \end{aligned}
\]
Cancelling $u_1^2$ gives the claim.

\begin{proposition}\label{prop:DT-coordinate-action}
There is a normalized folded cluster DT transformation $\Phi$ of the
fiber $\zeta=-1$ of \eqref{eq:zeta-family-exchanges} whose pullback satisfies
\begin{equation}
\label{eq:DT-coordinate-action}
 \Phi^*(u_i)=\udisc u_i\quad(i=1,5,7,8),
 \qquad
 \Phi^*(u_2,u_3,u_4,u_{13})=(u_2,u_3,u_4,u_{13}).
\end{equation}
\end{proposition}

\begin{proof}
Use the six-vertex unfolding with fold orbits
$\{1,4\},\{2,5\},\{3,6\}$.  In the mutable-row convention, all six rows
and the first six columns are mutable; the remaining columns correspond to
$u_2,u_3,u_4,u_{13},\zeta$:
\[
 \widetilde B_u=
 \begin{pmatrix}
 0&-1& 1& 0&-1& 1&-1& 1&-1& 0&0\\
 1& 0&-1& 1& 0&-1& 1&-1& 0& 1&1\\
 -1&1& 0&-1& 1& 0& 0& 0& 1&-1&0\\
 0&-1& 1& 0&-1& 1&-1& 1&-1& 0&0\\
 1& 0&-1& 1& 0&-1& 1&-1& 0& 1&1\\
 -1&1& 0&-1& 1& 0& 0& 0& 1&-1&0
 \end{pmatrix}.
\]
Zhou's DT word for this unfolding is the following (see \cite[Section~5.1]{Zhou}):
\[ \mathbf m=(1,3,2,4,6,5,1,6,4,3,2,5). \]
After mutation along $\mathbf m$ and terminal
relabeling in the order $(4,1,3,5,2,6)$, the first ten columns return to
their initial values and the $\zeta$-column changes sign.  Thus the terminal
coefficient is $\zeta^{-1}$; the ideal $(\zeta+1)$ is preserved, so the
transformation descends to the fiber $\zeta=-1$.

Applying the same twelve exchange mutations to the variables, identifying
each fold orbit, and setting $\zeta=-1$, gives, before normalization,
\[ u_1\longmapsto u_1V, \qquad u_5\longmapsto-u_5V, \qquad u_7\longmapsto u_7V, \]
where
\begin{equation}
\label{eq:DT-factor-V}
 V=
 \frac{
 (u_1^2u_2u_{13}+u_5^2u_2u_4+u_7^2u_3)^2
 -4u_1^2u_7^2u_2u_3u_{13}}
 {u_1^2u_5^2u_7^2}.
\end{equation}
The sign in the second coordinate is removed by the coefficient-torus
normalization which multiplies the second and fifth unfolded mutable
variables by $\zeta$ and multiplies $u_4$ by $\zeta^{-2}$.  The
exponent vector of this rescaling lies in the right kernel of the first ten
columns of $\widetilde B_u$, so it is an automorphism of the seed pattern.
On $\zeta=-1$ it changes only the second folded sign.

Finally, \eqref{eq:DT-two-identities} transforms \eqref{eq:DT-factor-V} into
\[ V=\frac{u_{10}^2-4u_2u_3u_{13}}{u_5^2}=\udisc. \]
This gives the first three identities in \eqref{eq:DT-coordinate-action}.
Applying $\Phi^*$ to
$u_7u_8=u_4u_5^2+u_1^2u_{13}$ gives the identity for $u_8$.
\end{proof}

Every homogeneous Laurent rational function $f$ of mutable degree $d$
therefore satisfies
\[ \Phi^*(f)=\udisc^d f. \]
In particular, $\Phi^*(\udisc)=\udisc^{-1}$.  It fixes the eight
degree-zero generators of $\Mu$, and by
\cref{prop:paired-theta-generators} it exchanges, up to nonzero scalars, the
four paired generators of degrees $1$ and $-1$.  Hence it preserves
$\Mu$.  This cluster transformation accounts for the weight-space pairing used in the
main proof.

\section{A \texorpdfstring{$D_5$}{D5}-parabolic interpretation of the seed}
\label{app:D5-parabolic}

This appendix interprets the seed with four frozen variables
\[ \mathbf u^\circ=(u_1,u_5,u_7\,;u_2,u_3,u_4,u_{13}) \]
using the maximal parabolic at the trivalent node of type $D_5$.  The
appendix is not used elsewhere.

\subsection{The trivalent parabolic of
\texorpdfstring{$\operatorname{Spin}_{10}$}{Spin(10)}}

Let
\[ V=\kk^3,\qquad A=V^\vee,\qquad B=C=\kk^2, \qquad W=B\otimes C. \]
Alternating forms on $B$ and $C$ define a nondegenerate symmetric form
on $W$ by
\[ \langle b\otimes c,b'\otimes c'\rangle_W =\omega_B(b,b')\omega_C(c,c'). \]
Equip $E=A\oplus W\oplus A^\vee$ with the split quadratic form pairing
$A$ with $A^\vee$ and restricting to this form on $W$.  Then
$G=\operatorname{Spin}(E)$ is of type $D_5$.  Let $P_3\subset G$
stabilize the isotropic three-plane $A$.  Deleting the trivalent node
$\alpha_3$ leaves $A_2\sqcup A_1\sqcup A_1$, so, up to a finite
central quotient,
\[ (L_{P_3})_{\mathrm{der}} \simeq\SL(A)\times\SL(B)\times\SL(C). \]
The standard parabolic grading \cite[Chapter~3]{CapSlovak} has
\[ \mathfrak g_0\simeq\mathfrak{gl}(A)\oplus\mathfrak{so}(W),\qquad \mathfrak g_{-1}\simeq A^\vee\otimes W\simeq V\otimes B\otimes C,\qquad \mathfrak g_{-2}\simeq\bigwedge^2A^\vee. \]
It may also be viewed as a Vinberg $\theta$-representation
\cite{VinbergTheta}.  Indeed, if $\gamma:\mathbb G_m\to G$ is the
grading cocharacter and $\epsilon$ is a primitive fifth root of unity,
then $\theta=\operatorname{Ad}(\gamma(\epsilon))$ has fixed Lie algebra
$\mathfrak g_0$, while $\mathfrak g_{-1}$ is its
$\epsilon^{-1}$-eigenspace.  Thus, up to the identity component and a
finite central quotient, $(L_{P_3},\mathfrak g_{-1})$ is a
$\theta$-group representation in Vinberg's sense.
The block decomposition of $\mathfrak{so}(A\oplus W\oplus A^\vee)$ gives
these three summands directly; the dimension check is
\[ 3+12+(9+6)+12+3=45=\dim\mathfrak{so}_{10}. \]
After identifying $\SL(A)$ with $\SL(V)$ by duality and choosing
compatible Borels, the Levi maximal unipotent subgroup acts as
$U_3\times U_2\times U_2$.  Hence
\[ \UR=\kk[\mathfrak g_{-1}]^{U_L},\qquad U_L=U_3\times U_2\times U_2. \]
The Kronecker grading is therefore the Levi highest-weight grading of the
degree $-1$ component.

\subsection{Orthogonal Pl\"ucker and Gram covariants}

The homogeneous space $G/P_3$ is $\operatorname{OGr}(3,10)$.  If $J$
is a Gram matrix of the form on $W$, the big cell of $G/P_3$ is parametrized by
\[
 (X,S)\in\operatorname{Hom}(A,W)\oplus\bigwedge^2A^\vee,
 \qquad
 \operatorname{colspan}
 \begin{pmatrix}
 I_3\\ X\\ S-\frac12X^{\mathsf T}JX
 \end{pmatrix}.
\]
The tensor representation is the slice $S=0$.  Order the ten rows as the
three $A$-rows, four $W$-rows, and three $A^\vee$-rows, and write
$p_{ijk}$ for the corresponding Pl\"ucker minor.  With compatible bases
and Borels, direct substitution identifies the seven seed functions, up to
nonzero scalars, as
\[
\begin{aligned}
 u_1&\doteq p_{234},&u_5&\doteq p_{456},&u_7&\doteq p_{468},\\
 u_2&\doteq p_{346},&u_3&\doteq p_{345},&u_4&\doteq p_{238},&
 u_{13}&\doteq p_{8,9,10}.
\end{aligned}
\]
Here $\doteq$ denotes equality up to an element of $\kk^\times$.  Two of
these identifications follow immediately from the big-cell matrix:
\[ p_{238}=-\frac12\langle v_1,v_1\rangle_W,\qquad p_{8,9,10}=\left(-\frac12\right)^3\det(X^{\mathsf T}JX). \]
Thus $p_{238}\doteq D_1$ and $p_{8,9,10}\doteq D_3$.

Write $X=(v_1,v_2,v_3)$, with $v_i\in W$, and put
\[ D_r=\det\bigl(\langle v_i,v_j\rangle_W\bigr)_{1\le i,j\le r}. \]
Up to the same scalar convention,
\[ u_4=D_1,\qquad u_9=D_2,\qquad u_{13}=D_3. \]
The pair $u_2,u_3$ consists of the two highest components of
$v_1\wedge v_2$ under
\[ \bigwedge^2(B\otimes C) \simeq (\Sym^2B\otimes\bigwedge^2C) \oplus(\bigwedge^2B\otimes\Sym^2C). \]
If $N=*(v_1\wedge v_2\wedge v_3)\in W$, then $u_5$ is a
highest-weight coordinate of $N$, while $u_7$ is a highest-weight
component of $v_1\wedge N$.  The frozen variables $u_2,u_3$ are
highest-weight coordinates in the two irreducible
$\operatorname{Spin}_4$-summands above; $u_4=D_1$ and $u_{13}=D_3$
are the first and third principal Gram determinants.  The additional
generator in $\UR=\Mu[u_9]$ is the second principal Gram determinant
$D_2$.  Under these identifications, the relation
\[ u_7u_8=u_4u_5^2+u_1^2u_{13} \]
is the corresponding relation among orthogonal Pl\"ucker and Gram
coordinates.

\subsection{The \texorpdfstring{$D_5$}{D5}-weight lattice and Gale duality}

Let $\Lambda_{322}$ be the lattice of integer triples
$(\lambda;\mu;\nu)$ with
$\lvert\lambda\rvert=\lvert\mu\rvert=\lvert\nu\rvert=n$.  Realize
\[ Q(D_5)=\{z\in\ZZ^5\mid z_1+\cdots+z_5\equiv0\pmod2\} \]
with simple roots
$\alpha_1=e_1-e_2,\alpha_2=e_2-e_3,\alpha_3=e_3-e_4,
\alpha_4=e_4-e_5,\alpha_5=e_4+e_5$.  The map
\[ \Psi(\lambda;\mu;\nu) =(-\lambda_3,-\lambda_2,-\lambda_1, \mu_1+\nu_1-n,\nu_1-\mu_1) \]
is an isomorphism $\Lambda_{322}\simeq Q(D_5)$.  Indeed, for
$z\in Q(D_5)$, take
\[ n=-(z_1+z_2+z_3),\quad \lambda=(-z_3,-z_2,-z_1),\quad \mu_1=\frac{n+z_4-z_5}{2},\quad \nu_1=\frac{n+z_4+z_5}{2}, \]
and then $\mu_2=n-\mu_1$, $\nu_2=n-\nu_1$.

For
\[ (x_1,\ldots,x_7)=(u_1,u_5,u_7,u_2,u_3,u_4,u_{13}), \]
let $\chi_i=\Psi(\wt(x_i))$.  The rows of the following matrix are the
seven $D_5$-weights, while $B_u$ is obtained from $B_{u,\zeta}$ by
deleting the auxiliary $\zeta$-column:
\[
 \Omega_{D_5}=
 \begin{pmatrix}
 0&0&-1&1&0\\
 -1&-1&-1&1&0\\
 -1&-1&-2&1&1\\
 0&-1&-1&1&1\\
 0&-1&-1&1&-1\\
 0&0&-2&0&0\\
 -2&-2&-2&0&0
 \end{pmatrix},
 \qquad
 B_u=
 \begin{bmatrix}
 0&-2& 2&-1& 1&-1& 0\\
 2& 0&-2& 1&-1& 0& 1\\
 -2& 2& 0& 0& 0& 1&-1
 \end{bmatrix}.
\]
Let $\chi:\ZZ^7\to Q(D_5)$ send $e_i$ to $\chi_i$.

\begin{proposition}[Gale duality for the seed]
\label{prop:D5-Gale-duality}
The exchange matrix is an integral Gale dual of the seven $D_5$-weights:
\[ B_u\Omega_{D_5}=0, \]
and there is an exact sequence
\[ 0\longrightarrow\ZZ \xrightarrow{\,1\mapsto(1,1,1)\,}\ZZ^3 \xrightarrow{\,B_u^{\mathsf T}\,}\ZZ^7 \xrightarrow{\,\chi\,}Q(D_5)\longrightarrow0. \]
\end{proposition}

\begin{proof}
The identities
\[
\begin{gathered}
 \alpha_1=-\chi_1-\chi_2+\chi_4+\chi_5,\quad
 \alpha_2=\chi_1-\chi_2+\chi_3-\chi_4,\quad
 \alpha_3=-\chi_1,\\
 \alpha_4=\chi_1+\chi_2-\chi_3,\quad
 \alpha_5=\chi_1+\chi_2-\chi_3+\chi_4-\chi_5
\end{gathered}
\]
show that $\chi$ is surjective.  Integral row reduction gives
\[ \ker\chi= \{(2s,2t,-2s-2t,s+t,-s-t,t,s)\mid s,t\in\ZZ\} =\operatorname{im}(B_u^{\mathsf T}). \]
Finally, $B_u^{\mathsf T}r=0$ forces $r_1=r_2=r_3$.
\end{proof}
The \texttt{d5} task in \path{Kron322.py} checks the
weight conversion, Smith normal form index, matrix product, ranks, integral kernel, and
Hermite normal form.

The three rows of $B_u$ give the homogeneity identities
\[ 2\chi_3+\chi_5=2\chi_2+\chi_4+\chi_6,\qquad 2\chi_3+\chi_5=2\chi_1+\chi_4+\chi_7, \qquad 2\chi_2+\chi_6=2\chi_1+\chi_7. \]

\begin{center}
\scriptsize
\renewcommand{\arraystretch}{1.08}
\begin{tabular}{@{}>{\raggedright\arraybackslash}p{0.28\textwidth}
                    >{\raggedright\arraybackslash}p{0.65\textwidth}@{}}
\toprule
\textbf{Cluster datum}&\textbf{Type-$D_5$ interpretation}\\
\midrule
$3\times2\times2$ tensor space
& The degree $-1$ component $\mathfrak g_{-1}$ of
  $P_3\subset\operatorname{Spin}_{10}$.\\
Triple $U$-invariant algebra $\UR$
& The Levi highest-weight algebra $\kk[\mathfrak g_{-1}]^{U_L}$.\\
Seven seed functions
& Restricted orthogonal Pl\"ucker covariants on
  $\operatorname{OGr}(3,10)$.\\
Frozen pair $u_2,u_3$
& Highest-weight coordinates in the two irreducible summands of
  $\bigwedge^2(B\otimes C)$.\\
Frozen variables $u_4,u_{13}$
& The first and third principal Gram determinants $D_1,D_3$;
  $u_9=D_2$ is the second one.\\
Exchange data
& The rows of $B_u$ are the full lattice of $D_5$-weight relations;
  the third exchange is the Pl\"ucker--Gram identity.\\
Mutable Markov quiver
& The mutation class of the seed; it is of infinite cluster type, not of
  finite Dynkin type $D_5$.\\
\bottomrule
\end{tabular}
\end{center}

The $D_5$ description identifies the ambient representation, the rank-five
weight lattice, and the homogeneity relations.  The mutation class remains
the Markov class.

\section{Two further cluster models}
\label{app:other-models}

The two models below are stated without proof and are not used in the main
argument.  The invariants $u_1,\ldots,u_{13}$ retain the normalization of
\eqref{eq:basic-relations}; the formulas fix the primed variables.
A negative exchange sign denotes specialization at $\zeta=-1$ in an
ordinary cluster family.  The matrices omit the auxiliary $\zeta$-column.

\subsection{A model with four mutable variables}

Take mutable variables $(u_1,u_2,u_3,u_6)$ and frozen variables
$(u_7,u_8,u_9)$.  The exchange relations are
\begin{align*}
 u_1u_1'&=u_2u_8+u_3u_7,
 &u_2u_2'&=u_1^2u_9-u_6^2,\\
 u_3u_3'&=u_1^2u_9-u_6^2,
 &u_6u_6'&=u_3u_7-u_2u_8,
\end{align*}
where
\[ u_1'=u_{10},\qquad u_2'=4u_3u_4,\qquad u_3'=4u_2u_4,\qquad u_6'=u_5. \]
These formulas follow from the first three identities in
\eqref{eq:basic-relations}.  In the mutable-row convention, with columns ordered
by $(u_1,u_2,u_3,u_6;u_7,u_8,u_9)$, the exchange matrix is
\[
 B_A=
 \begin{bmatrix}
 0&-1& 1& 0& 1&-1& 0\\
 2& 0& 0&-2& 0& 0& 1\\
-2& 0& 0& 2& 0& 0&-1\\
 0& 1&-1& 0&-1& 1& 0
 \end{bmatrix}.
\]
For this model,
\[ \UR=\mathcal M_A[u_4,u_{13}]. \]

\subsection{A model with five mutable variables}

Take mutable variables $(u_1,u_2,u_3,u_5,u_6)$ and frozen variables
$(u_{11},u_{12})$.  With the normalizations
\begin{equation*}
 \begin{aligned}
 u_1'&=u_5u_9,&
 u_2'&=-2u_3u_4u_5-u_6u_8,&
 u_3'&=2u_2u_4u_5-u_6u_7,\\
 u_5'&=u_1u_9,&
 u_6'&=-u_{10},
 \end{aligned}
\end{equation*}
\eqref{eq:basic-relations} gives the exchange relations:
\begin{align*}
 u_1u_1'&=u_2u_{12}-u_3u_{11},
 &u_2u_2'&=u_1u_3u_{11}+u_5u_6^2,\\
 u_3u_3'&=u_1u_2u_{12}-u_5u_6^2,
 &u_5u_5'&=u_2u_{12}-u_3u_{11},\\
 u_6u_6'&=u_2u_{12}+u_3u_{11}.
\end{align*}
With columns ordered by
$(u_1,u_2,u_3,u_5,u_6;u_{11},u_{12})$, the exchange matrix is
\[
 B_B=
 \begin{bmatrix}
 0& 1&-1& 0& 0&-1& 1\\
-1& 0&-1& 1& 2&-1& 0\\
 1& 1& 0&-1&-2& 0& 1\\
 0&-1& 1& 0& 0& 1&-1\\
 0&-1& 1& 0& 0& 1&-1
 \end{bmatrix}.
\]
For this model,
\[ \UR=\mathcal M_B[u_4,u_9,u_{13}]. \]

The corresponding cones and theta-generator pairings lead to analogues of
\cref{thm:polyhedral-Kronecker,thm:closed-formula}.  Their proofs and finite
sums are omitted.  The model used in the main text has the shortest
presentation.

\bibliographystyle{amsalpha}
\bibliography{Kron322}

\end{document}